\documentclass[a4paper]{article}
\date{}
\usepackage{myStylePaper}

\newcommand{\Debug}{0}

\newcommand{\comment}[1]{}
\newcommand{\mymargin}[1]{
 \ifnum \Debug = 1
  \marginpar{%
    \begin{minipage}{\marginparwidth}\small%
      \begin{flushleft}%
        {\color{blue}#1}%
      \end{flushleft}%
   \end{minipage}%
  }%
 \fi
}%

\ifnum \Debug = 1 \usepackage[notref,notcite]{showkeys}
\fi

\newcommand{\Cg}{Cayley graph}

\newcommand{\labtequ}[2]{
 \begin{equation} \label{#1} 	\begin{minipage}[c]{0.9\textwidth}  #2 \end{minipage} \ignorespacesafterend \end{equation} }

\newcommand{\Tr}[1]{Theorem~\ref{#1}}

\newcommand{\Sr}[1]{Section~\ref{#1}}

\newcommand{\Prr}[1]{Pro\-position~\ref{#1}}

\newcommand{\Cr}[1]{Corollary~\ref{#1}}

\newcommand{\Dr}[1]{De\-fi\-nition~\ref{#1}}

\newcommand{\pp}{partite presentation}
\newcommand{\bpp}{2-partite presentation}
\newcommand{\pcg}{partite \Cg}
\newcommand{\bpcg}{2-partite \Cg}

\newcommand{\pf}{partition-friendly}
\newcommand{\Spl}{PCay}

\title{Presentations for Vertex Transitive Graphs}

\author{
	Agelos Georgakopoulos\thanks{Supported by the European Research Council (ERC) under the European Union's Horizon 2020 research and innovation programme (grant agreement No 639046).} \textsuperscript{1} and Alex Wendland\textsuperscript{1,3}\\
	\newline\\
	\textsuperscript{1}Department of Mathematics, University of Warwick, Coventry, CV4 7AL, UK\\
	\textsuperscript{3}A.P.Wendland@warwick.ac.uk
}

\begin{document}

\maketitle
	
\begin{center}{\it With an appendix by 	Matthias Hamann and Alex Wendland}
\end{center}

\begin{abstract}
	We generalise the standard constructions of a \Cg\ in terms of a group presentation by allowing some vertices to obey different relators than others. The resulting notion of presentation allows us to represent every vertex transitive graph. As an intermediate step, we prove that every countably infinite, connected, vertex transitive graph has a perfect matching.
	
	Incidentally, we construct an example of a 2-ended cubic vertex transitive graph which is not a Cayley graph, answering a question of Watkins from 1990. 
\end{abstract}

\section{Introduction}
Every \Cg\ is (vertex-)transitive but the converse is not true, with the Petersen graph being a well-known example. A lot of research focuses on understanding how much larger the class of transitive graphs is or, what is essentially the same, on extending results from \Cg s to transitive graphs, see e.g.\ \cite{EsFiWhCoa,Le83,Mw09,Wa90} and references therein. 
This paper offers a new algebraic way of defining graphs, which we will prove to have the power to present all transitive graphs. 

The idea is to still define our graphs by means of generators and relators similarly to \Cg s defined via group presentations, but we now allow different vertices to obey different sets of relators. The fewer `types' of vertices we have the closer our graph is to being a \Cg. This is perhaps best explained with an example: in Figure~\ref{petersen-intro} we have directed and labelled the Petersen graph with two letters $r$ and $b$ that make it look almost like a \Cg. But a closer look shows that if we start at any exterior vertex $v$ and follow a sequence of edges labelled $brbrr$ then we return to $v$, while this is not true if $v$ is one of the interior vertices. In that case, $brrbr$ is an example of a word that gives rise to a cycle.

	\showFigTikz{
		\begin{tikzpicture}[scale = 0.7]
			\node (G-11) at (-2,0) [circle,fill=black,scale = 0.5] {};
			\node (G-12) at (2,0) [circle,fill=black,scale = 0.5] {};
			\node (G-13) at (3,3) [circle,fill=black,scale = 0.5] {};
			\node (G-14) at (0,5) [circle,fill=black,scale = 0.5] {};
			\node (G-15) at (-3,3) [circle,fill=black,scale = 0.5] {};
			\node (G-21) at (-1,1) [rectangle,fill=black,scale = 0.5] {};
			\node (G-22) at (1,1) [rectangle,fill=black,scale = 0.5] {};
			\node (G-23) at (2,3) [rectangle,fill=black,scale = 0.5] {};
			\node (G-24) at (0,4) [rectangle,fill=black,scale = 0.5] {};
			\node (G-25) at (-2,3) [rectangle,fill=black,scale = 0.5] {};
			\draw[red,->] (G-11) edge (G-12);
			\draw[red,->,thick] (G-12) edge (G-13);
			\draw[red,->,thick] (G-13) edge (G-14);
			\draw[red,->] (G-14) edge (G-15);
			\draw[red,->] (G-15) edge (G-11);
			\draw[red,<-] (G-21) edge (G-23);
			\draw[red,<-] (G-23) edge (G-25);
			\draw[red,<-] (G-25) edge (G-22);
			\draw[red,<-,thick] (G-22) edge (G-24);
			\draw[red,<-] (G-24) edge (G-21);
			\draw[blue,<->] (G-11) edge (G-21);
			\draw[blue,<->,thick] (G-12) edge (G-22);
			\draw[blue,<->] (G-13) edge (G-23);
			\draw[blue,<->,thick] (G-14) edge (G-24);
			\draw[blue,<->] (G-15) edge (G-25);
		\end{tikzpicture}}{petersen-intro}{
		The Petersen graph, labelled by two letters $r$ (for red) and $b$ (for blue). The cycle obtained  by reading the `relation' $rbrrb$ starting at the top square vertex is depicted in bold lines.}
		
This example motivates our definition of a \defin{\pp}, which prescribes a number of types of vertices, and a set of relators for each type. Moreover, it entails a set of generators, and for each generator $s$ it prescribes the type of end-vertex of an edge labelled $s$ for each type of starting vertex. 
The precise definition of partite presentations in the case where there are only two types of vertices, which we call 2-partite presentations, is given in \Sr{sec special}. The case with more classes is more involved, and it is given in \Sr{sec general}.

We show how each \pp\ defines a graph, by imitating the standard definitions of a \Cg\ via a group presentation: either as a quotient of a free group by the normal subgroup generated by the relators (Definition~\ref{specgrpdef}), or as the 1-skeleton of the universal cover of the presentation complex (Definition~\ref{spectopdef}). The resulting \defin{\pcg} is always regular, with vertex-degree determined by the generating set, and it admits a group of automorphisms acting on it regularly and with as many orbits as the number of types of vertices prescribed by its presentation (\Prr{genloctrans}). In particular, \bpp s always give rise to bi-Cayley graphs. We prove this, as well as a converse statement, in \Sr{sec Bi-Cayley}.

Our first main result says that our formalism of \pcg s is general enough to describe all vertex transitive graph:
\begin{thm} \label{finite-vertex-presentation}
	Every 
	countable, vertex transitive, graph has a \pp.\\
\end{thm}
In general, for the proof of this we allow for the vertex types to be in bijection with the vertex set of the graph in question. It would be interesting to study how much the number of vertex types can be reduced, see \Sr{sec Conc}. As we remark there, there are vertex transitive graphs that require infinitely many vertex types in any \pp; most Diestel-Leader  graphs \cite{DiLeCon} have this property. 
In the converse direction, we show, in \Sr{sec Line}, that every line graph of a Cayley graph $\Gha$ admits a \pp\ with at most as many vertex types as the number of generators of $\Gha$. 

The proof of \Tr{finite-vertex-presentation} involves decomposing the edge-set into cycles. This decomposition is not obvious, and it is related to a conjecture of Leighton \cite{Le83} disproved by Maru\v{s}i\v{c} \cite{Ma81}; see \Sr{sec mcc} for more. To find such decomposition we had to generalise a result of \cite[Theorem 3.5.1]{GR01}, saying that every connected finite vertex transitive graph has a matching that misses at most one vertex, to infinite vertex transitive graphs, which might be of independent interest:
\begin{thm} \label{inf-vertex-matching}
	Every countably infinite, connected, vertex transitive graph has a perfect matching.
\end{thm}
This result is proved in the Appendix, which can be read independently. The locally finite case had previously been obtained by Leemann \cite{Leemann}.
 
Incidentally,  we find a cubic (i.e.\ 3-regular) 2-ended vertex transitive graph which is not a Cayley graph, answering a question of Watkins \cite{Wa90}, recently revived by Grimmett \& Li \cite{GL16}. Although this construction does not explicitly use the theory developed in this paper, our study of \pp s helped us understand where to look for such examples.

\begin{thm} \label{cubic-2-ended-vertex-transitive-example}
	There exists a cubic 2-ended vertex transitive graph which is not a Cayley graph.
\end{thm}
This is proved in section~\ref{sec 2-ended}, which  can again be read independently. All graphs in this paper (excluding the Appendix) are assumed to be locally finite, i.e.\ with finite vertex degrees, although much of our work could extend to the general case.

\section{Preliminaries}

\subsection{Graphs and automorphisms}
We work with the notion of graph as defined by Gersten \cite{Ge83}. A \emph{graph} $\Gha$ comprises a set of \defin{vertices} $V(\Gha)$, and a set of \defin{directed edges} $\overrightarrow{E}(\Gha)$, endowed with a fix point free involution $\inv: \overrightarrow{E}(\Gha) \rightarrow \overrightarrow{E}(\Gha)$ and a terminus map $\tau: \overrightarrow{E}(\Gha) \rightarrow V(\Gha)$. 
Sometimes we will express the elements of $\overrightarrow{E}(\Gha)$ as directed pairs $(v,w)$ with $v,w\in V(\Gha)$, in which case we tacitly mean that $\tau((v,w))=w$ and $(v,w)^{-1}=(w,v)$.

A directed edge $e\in \overrightarrow{E}(\Gha)$ is a \defin{loop}, if $\tau(e)= \tau(e^{-1})$. The \defin{degree} $d(v)$ of a vertex $v\in V(\Gha)$ is the cardinality of $|\tau^{-1}(v)|$. We say that $\Gha$ is $k$-regular if $d(v)=k$ for every $v\in V(\Gha)$.

To a graph $\Gha$ we can associate the set of undirected edges, $E(\Gha) := \overrightarrow{E}(\Gha)/\inv$. Thus $|\overrightarrow{E}(\Gha)|= 2|E(\Gha)|$ for every graph $\Gha$. 

Note that although we are talking about `directed edges', we are not talking about `directed graphs' in the sense of e.g.\ \cite{DiestelBook05}. Our edges can be thought of as undirected pairs of vertices, but our formalism allows us to distinguish between two orientations for each of them. Moreover, our formalism allows for multiple edges between the same pair of vertices, and multiple loops at a single vertex. Thus the pair $(V(\Gha), E(\Gha))$ is a multigraph in the sense of  \cite{DiestelBook05}.

A \defin{map of graphs} $\phi: \Gha \rightarrow \Ghb$ is a pair of maps $(\phi_V: V(\Gha) \rightarrow V(\Ghb), \phi_E: \overrightarrow{E}(\Gha) \rightarrow \overrightarrow{E}(\Ghb))$ where $\phi_E$ commutes with $\inv$ and $\phi_V \circ \tau = \tau \circ \phi_E$. For a graph $\Gha$, an \defin{endomorphism} is a map from $\Gha$ to itself, and it is called an \defin{automorphism} if $\phi_V$ and $\phi_E$ are bijections. The sets of these maps are denoted $\End(\Gha)$ and $\Aut(\Gha)$ respectively.

We say that $\Gha$ is \defin{vertex transitive} if $Aut(\Gha)$ acts transitively on $V(\Gha)$, and 
\defin{edge transitive} if $Aut(\Gha)$ acts transitively on $E(\Gha)$. We say that $\Gha$ is \defin{arc-transitive}, or \defin{symmetric}, if $Aut(\Gha)$ acts transitively on $\overrightarrow{E}(\Gha)$. We say $\Gha$ is \defin{semi-symmetric} if it is edge transitive and regular but not vertex transitive.

Given a set of undirected edges $S \subset E(\Gha)$ of a graph $\Gha$ an \defin{orientation} of $S$ is a subsets $O_S \subset \overrightarrow{E}(\Gha)$ such that $O_S/\inv = S$, and $O_S \cap O_S^{-1} = \emptyset$.

A \defin{walk} in $\Gha$ is an alternating sequence $v_0 e_1 v_1 \ldots e_n v_n$ of vertices and directed edges such that $\tau(e_i) = v_i$ and $\tau(e_i^{-1}) = v_{i-1}$ for every $1\leq i \leq n$.

\subsection{Groups and Cayley graphs}

Given a group $G$ and a subset $S \subset G$, we define the (right) \defin{\Cg} $\Gamma= \Cay(G,S)$ to be the graph with vertex set $V(\Gamma)= G$ and directed edge set $\overrightarrow{E}(\Gamma)= \{ (g,(gs)) \mid g\in G, s\in S\}$. Unless otherwise stated, we are not assuming that $S$ generates $G$, so that  \Cg s in this paper are not always connected. The group $G$ acts on $\Gamma$ by automorphisms, by multiplication on the left. 



\subsection{Colourings}
In this work a graph colouring will always refer to a colouring of the edges. A \defin{colouring of the undirected edges} of $\Gha$ is a map $c: E(\Gha) \rightarrow X$ whereas a \defin{colouring of the directed edges} is a map $c: \overrightarrow{E}(\Gha) \rightarrow X$, where $X$ is an arbitrary set called the set of colours.  
\comment{
If $S$ induces a 2-regular subgraph the orientation is assumed to be chosen so that $O_S$ forms a permutation on $V(\Gha \cap S)$. If we have a colouring $c: E(\Gha) \rightarrow X$ so that $c^{-1}(x)$ is 2-regular for all $x \in X$ via choice of orientation $O_x$ for each $c^{-1}(x)$ we can recover a colouring $c': \overrightarrow{E}(\Gha) \rightarrow \{x, x^{-1} \vert x \in X\}$ where
\[
c'(e) = \begin{cases} x & e \in O_x\\ x^{-1} & e^{-1} \in O_x \end{cases}.
\] 
Similarly if we have a colouring $c': \overrightarrow{E}(\Gha) \rightarrow X$ such that there is a bijection $\inv: X \rightarrow X$ such that $c'(e)^{-1} = c'(e^{-1})$ then one can define colouring $c: E(\Gha) \rightarrow X/\inv$.
}

\subsection{Covering spaces} \label{sec covers}

	A covering space (or cover) of a topological space $X$ is a topological space $C$ endowed with a continuous surjective map $\psi: C \rightarrow X$ such that for every $x \in X$, there exists an open neighbourhood $U$ of $x$, such that $\psi^{-1}(U)$ is the union of disjoint open sets in $C$, each of which is mapped homeomorphically onto $U$ by $\psi$.

Given a map of spaces $\phi: Y \rightarrow X$, and point $y \in Y$ such that $p(y) = x$, we obtain an induced map on the level of fundamental groups $\phi_{\ast}: \pi_1(Y,y) \rightarrow \pi_1(X,x)$  by composition. For a covering map $\phi$ we know that $\phi_{\ast}$  is injective \cite[Proposition 1.31]{HA02}. If $C$ is arc-connected, and $\pi_1(C,c) = 1$, i.e.\ $C$ is simply connected, we call $C$ the universal cover, which is well-known to be unique when it exists.

Given a cover $\psi: C \rightarrow X$ and a map $\phi: Y \rightarrow X$ (with $Y$ path connected and locally path connected) we obtain a lift $\tilde{\phi}: Y \rightarrow C$ (where $\phi = \psi \circ \tilde{\phi}$) of $\phi$ if and only if $\phi_{\ast}(\pi_1(Y,y)) \subset  \psi_{\ast}(\pi_1(C,c))$ \cite[Proposition 1.33]{HA02}. Moreover for any preimage $c \in \psi^{-1}(x)$ we can choose $\tilde{\phi}(y) = c$. 

Lastly we recall the classification of covering spaces:
	
\begin{thm} \label{cover-correspondence}
	(Hatcher \cite[Theorem 1.38]{HA02}) Let $X$ be a path-connected, locally path-connected, and semilocally simply-connected topological space. Then there is a bijection between the set of isomorphisms classes of path-connected covering spaces $\psi: C \rightarrow X$ and the set of subgroups (up to conjugation) of $\pi_1(X)$, obtained by associating the subgroup $\psi_{\ast}(\pi_1(C))$ to the covering space $C$. 
\end{thm}



\section{\bpp s} \label{sec special}

\subsection{Algebraic definition}
We start by recalling one of the standard definitions of a Cayley graph, in order to then adapt it into the definition of a \bpcg. 

Let $\Gpa$ be a group. A \defin{presentation} $\langle \sS \vert \sR \rangle$ of $\Gpa$ consists of a generating set $\sS \subset \Gpa$ and a \defin{relator set} $\sR \subset F_{\sS}$ such that $F_{\sS} / \langle \langle \sR \rangle \rangle = G$,  where $F_{\sS}$ denotes the free group with free generating set   $\sS$. For a group presentation $\langle \sS \vert \sR \rangle$, we can construct the \defin{Cayley graph} $\Cay\langle \sS \vert \sR \rangle$ in the following manner. Let 
$T_{\sS}$ be the $2\vert \sS \vert$-regular tree defined by  
\begin{align*}
V(T_{\sS}) & := F_{\sS}, \mbox{and }\\
\overrightarrow{E}(T_\sS) & := \{(w, ws) \vert w \in F_{\sS}, s \in \sS \cup \sS^{-1}\}.
\end{align*}
We endow $T_{\sS}$ with a colouring $c: \overrightarrow{E}(T_{\sS}) \rightarrow \sS \cup \sS^{-1}$ defined by $c(w,ws) = s$ and $c(ws,w) = s^{-1}$. 
Let $R := \langle \langle \sR \rangle \rangle$ be the normal closure of $\sR$ in $F_{\sS}$. Define an equivalence relation $\sim$ on $V(T_{\sS}) = F_{\sS}$ by letting $v \sim w$ whenever $v^{-1}w \in R$. Extend $\sim$ to $\overrightarrow{E}(T_{\sS})$ by demanding $e \sim d$ whenever $c(e) = c(d)$ and $\tau(e) \sim \tau(d)$ and $\tau(e^{-1}) \sim \tau(d^{-1})$. Then $\Cay\langle \sS \vert \sR \rangle$ can be defined as the quotient $T_{\sS}/\sim$. The corresponding  covering map is denoted by $\eta: T_{\sS} \rightarrow \Cay\langle \sS \vert \sR \rangle$. Note that as $\sim$ preserves $c$, we obtain a unique colouring $c': \overrightarrow{E}(\Cay\langle \sS \vert \sR \rangle) \rightarrow \sS \cup \sS^{-1}$ satisfying $c = c' \circ \eta$.

\medskip
This definition of the Cayley graph is standard. All Cayley graphs defined this way have even degrees:  involutions in $\sS$ give rise to pairs of `parallel' edges with the same end-vertices. However, in certain contexts it is desirable to replace such  pairs of parallel edges by single edges. To accommodate for this modification ---which is important for us later as we want to capture odd-degree graphs such as the Petersen graph with our presentations--- we now introduce  \defin{\invo\ presentations and Cayley graphs}

For a group presentation $\langle \sS \vert \sR' \rangle$, we define the \defin{\invo\ presentation} $P = \langle \sU, \sI \vert \sR \rangle$ where\\ $\sI := \{s \in \sS : s^2 \in \sR'\}$ and $\sU:= \sS \backslash \sI$. Define the corresponding \defin{\invo\ free group}\\ $MF_P := \langle \sS \vert \{s^2 : s \in \sI\} \rangle$. (Thus $MF_P$ is a free product of infinite cyclic groups, one for each $s\in \sU$, and cyclic groups of order 2,  one for each $s\in \sI$.) Let $\phi: F_{\sS} \rightarrow MF_P$ be the unique homomorphism extending the identity on $F_{\sS}$, as provided by the universal property of free groups, and let $\sR := \phi(\sR') \backslash \{1\} \subset MF_P$. Define the $\vert \sS \cup \sS^{-1} \vert$-regular tree  $T_{P}$ by 
\begin{align*}
V(T_{P}) & := MF_P\\
\overrightarrow{E}(T_P) & := \{(w, ws) \vert w \in MF_P, s \in \sS \cup \sS^{-1}\}.
\end{align*}
We proceed as above to define the colouring $c$ and the relation $\sim$, and obtain the  \defin{\invo\ Cayley graph} as the quotient $T_{P}/\sim$.


\medskip
We now modify the above construction of the Cayley graph, to obtain our \pcg s. The basic idea is to partition the vertex set into two (and later more than two) classes $V_1,V_2$, obeying different sets of relators $\sR_0, \sR_1$. This bipartition creates the need to distinguish our generators too into two classes $\sS_1, \sS_2$, the former corresponding to edges staying in the same partition class, and the latter corresponding to edges incident with both classes $V_1,V_2$.

We will formally define a \defin{\bpp} as a 4-tuple $P = \langle \sS_1, \sS_2 \vert \sR_0, \sR_1 \rangle$, and explain how this data is used to define a \pcg\ $\Sp(P)$, in analogy with the above definition of a Cayley graph $\Cay\langle \sS \vert \sR \rangle$ corresponding to a group presentation $P = \langle \sS \vert \sR \rangle$.  The set $\sS_1$ is an arbitrary set of `generators'. We partition $\sS_2$ into two sets, $\sS_2= \{\sU,\sI\}$, so that $\sS_1, \sU,\sI$ are pairwise disjoint. Their union $\sS:=  \sS_1 \cup \sU \cup \sI$  will be our set of \defin{generators}. The necessity of distinguishing $\sS_2$ into $\sU,\sI$ is to allow for some involutions, namely the elements of $\sI$, to give rise to single edges in our graphs, just like in the above definition of \invo\ Cayley graph. 

As in our definition of \invo\ Cayley graph, we let $MF_P := \langle \sS \vert \{s^2 : s \in \sI\} \rangle$. Let $\vert \cdot \vert_{\sS_2}$ be the unique homomorphism from $MF_P$ to $\bZ/2\bZ$ extending 
\[
\vert s \vert_{\sS_2} = \begin{cases} 0 & \mbox{ if } s \in \sS_1\\ 1 & \mbox{ if } s \in \sS_2 \end{cases}.
\]
We have that $K := Ker(\vert \cdot \vert_{\sS_2})$ is an index-two subgroup of $MF_P$, and so its cosets $\tilde{V_1}:= K$ and $\tilde{V_2} := \sS_2 K$ bipartition $MF_P$. 

\begin{defn} \label{spec sp pres def}
For any two sets $\sR_0, \sR_1 \subset K$, called \defin{relator sets}, we call the tuple $\langle \sS_1, \sS_2 \vert \sR_0, \sR_1 \rangle$ a \defin{\bpp}. 
\end{defn}
(The restriction $\sR_i \subset K$ does not have an analogue in the definition of Cayley graph; the intuition is that relators should start and finish at the same side of the bipartition $V_1,V_2$ because they are supposed to yield cycles in the graph.)

Given a \bpp\ $P = \langle \sS_1, \sU, \sI \vert \sR_0, \sR_1 \rangle$, recall that $MF_P = \langle \sS \vert \{ s^2 : s \in \sI \} \rangle$, and define the ($\vert \sS \cup \sS^{-1} \vert$-regular) tree $T_P$ by
	\begin{align*}
		V(T_P) & := MF_P\\
		\overrightarrow{E}(T_P) & := \{(w, ws) \vert w \in MF_P, s \in \sS \cup \sS^{-1}\}.
	\end{align*}
Define the subgroups
	\begin{align*}
		R_0 & \mbox{ to be the normal closure of } \sR_0 \cup \{ s r s^{-1} :  r \in \sR_1, \ s \in \sS_2\} \mbox{ in } K, \mbox{ and }\\
		R_1 & \mbox{ to be the normal closure of } \sR_1 \cup \{ s r s^{-1} :  r \in \sR_0, \ s \in \sS_2\}  \mbox{ in } K.
	\end{align*}
Here $R_i \leq K \leq MF_P$ is the analogue of the normal subgroup $R$ of $MF_P$ in the  definition of $\Cay(P)$, but now having two versions corresponding to our two classes of elements of $MF_P$, namely $\{\tilde{V_1}, \tilde{V_2}\} := \{K, \sS_2K \}$. In analogy with the relation $\sim$ above, we now write $v \sim w$ whenever $v^{-1}w \in R_i$ for $v,w \in \tilde{V_i}$. We extend $\sim$ to the edges of $T_P$ via $e \sim d$ if $c(e) = c(d)$, $\tau(e) \sim \tau(d)$, and $\tau(e^{-1}) \sim \tau(d^{-1})$. 
\begin{defn} \label{specgrpdef}
The \defin{\bpcg} $\Spx = \Sp(P) =: \Gha$ is the quotient $T_P/\sim$. 
\end{defn}
The edge set of $\Gha$ can thus be written as $\overrightarrow{E}(\Gha) = \overrightarrow{E}(T_P)/\sim$. 

As before, we have a natural colouring $c: \overrightarrow{E}(T_P) \rightarrow \sS \cup \sS^{-1}$ defined by $c(w, ws) = s$, and as $\sim$ preserves $c$, the latter factors into  $c': \overrightarrow{E}(\Gha) \rightarrow \sS \cup \sS^{-1}$, i.e.\ the  unique colouring satisfying $c = c' \circ \eta$ where again $\eta$ denotes the projection map corresponding to $\sim$.

Note that this is a generalisation of the \invo\ Cayley graph. When $\sI=\emptyset$ we have a  generalisation of the standard Cayley graph. 

Borrowing terminology from groupoids, we define the \defin{vertex groups} of our \pp\ to be $G_i := K / R_i $ for $i \in \bZ/2\bZ$. 

The condition $\sR_i \subset K$ implies that if $v \sim w$ then $v$ and $w$ belong to the same coset $\tilde{V_0}$ or $\tilde{V_1}$ of $K$ in $MF_P$ by the definitions. Thus factoring by $\sim$ projects the bipartition $\{\tilde{V_0}, \tilde{V_1}\}$ of $MF_P$ into a bipartition $\{{V_0}, {V_1}\}$ of $V(\Gha)$, with  $V_i:= \tilde{V}_i/\sim$. It follows from these definitions that $G_i$ is in canonical bijection with $V_i$.

As in the case of Cayley graphs, relators in the presentation yield closed walks in $\Gha$, but now we need to start reading our relators at the right side of the bipartition for this to be true: for every $i \in \bZ/2\bZ$ and each $r\in \sR_i$ and $v\in V_i$, if we start at $v$ and follow the directed edges of $\Gha$ with colours dictated by $r$ one-by-one, we finish our walk at $v$.

We now explain how the Petersen graph can be obtained as a \bpcg:
\begin{examp} \label{ex Pet 1}
	 Theorem \ref{Petersen-Split-Presentation} below asserts that the Petersen graph $P(5,2)$ (Figure~\ref{petersen-intro}) is isomorphic to\\ $\Sp(\langle \sS_1 = \{a\}, \sU = \emptyset, \sI = \{b\} \vert \sR_0 = \{a^5, aba^2b, b^2\}, \sR_1 = \{a^5\}\rangle) = \Sp\langle \{a\}, \emptyset, \{b\} \vert \{a^5, aba^2b\}, \{a^5\} \rangle$. For this presentation we have  
\begin{itemize}
\item $MF_P = \langle a, b \vert b^2 \rangle$, so that $T_P$ is the 3-regular tree; 
\item $K = \langle \langle a, bab \rangle \rangle \leq MF_P$; 
\item  $R_0 = \langle \langle a^5, aba^2b, ba^5b \rangle \rangle_{K}$, and 
\item  $R_1 = \langle \langle ba^5b, baba^2, a^5 \rangle \rangle_{K}$. 
\end{itemize}

The vertex groups $G_i = K / R_i$ are generated by any generating set of $K$, in particular by  $\{a,bab\}$. They abide by the relations that generate $R_i$ so in the case of $R_0$ these are $a^5$, $aba^2b = a(bab)^2$ and $ba^5b = (bab)^5$ (when we write them in terms of the generators of $K$). So we have 
	\begin{align*}
		G_0 = & \langle a, bab \vert a^5, a(bab)^2, (bab)^5 \rangle\\
		= & \langle bab \vert (bab)^{-10}, (bab)^5 \rangle & \mbox{as } a = (bab)^{-2}\\
		= & \bZ/5\bZ = \langle bab \rangle
	\end{align*}
	and similarly
	\begin{align*}
		G_1 = & \langle a, bab \vert (bab)^5, (bab)a^2, a^5 \rangle\\
		= & \langle a \vert a^{-10}, a^5 \rangle & \mbox{as } (bab) = a^{-2}\\
		= & \bZ/5\bZ = \langle a \rangle.
	\end{align*}
The fact that $G_0$ is isomorphic to $G_1$ is not a coincidence as we remark at the end of this section. In Figure~\ref{petersen-intro}, the  vertices depicted as square correspond to $V_0= \tilde{V}_0/\sim$,  and vertices depicted as circles correspond to $V_1= \tilde{V}_1/\sim$.
\end{examp}

Note that we have made $\sS$ a subset of the group $MF_P$, and so each $s\in \sS$ has an inverse $s^{-1}$ in $MF_P$. With these inverses in mind we define $\sS^{-1}:= \{s^{-1} : s\in \sS\}$. Note that $s = s^{-1} $ exactly when $s\in \sI$. Moreover, as $\sS_1 \subset K$ and $G_i = K/R_i$, we can think of $\sS_1$ as a subset of $G_i$ in the following proposition:

\begin{prop} \label{Cay G0}
	For every \bpp\ $P = \langle \sS_1, \sU, \sI \vert \sR_0, \sR_1 \rangle$, the subgraph of $\Gha := \Sp(P)$ with edges coloured by $\sS_1 \cup \sS_1^{-1}$ is isomorphic to the disjoint union of $\Cay(G_0, \sS_1)$ and $\Cay(G_1, \sS_1)$.
\end{prop}
\begin{proof}
	Let $T_i$ be the subgraph of $T_P$ induced by the vertices of $\tilde{V}_{i}$, and $\Gha_i$ be the subgraph of $\Gha$ induced by $V_i = \tilde{V}_i/\sim$. We will show that $\Gha_i$ is isomorphic to $\Cay(G_i, \sS_1)$. 
	
	To begin with, recall that $\tilde{V}_0=K$ and $G_0=K / R_0$, and so $V_0$ is canonically identified with $G_0$. Thus to show that $\Gha_0$ is isomorphic to $\Cay(G_0, \sS_1)$, we need to check that $(v,w)$ is a directed edge of $\Gha_0$ coloured $s$ whenever  $w=v s$. The latter holds whenever $v's \in \eta^{-1}(w)$ for every $v'\in \eta^{-1}(v)$, which is exactly when $(v',v's)$ is  a directed edge of $T_P$ coloured $s$. This in turn is equivalent to $(v,w)$ being a directed edge of $\Gha_0$ coloured $s$ because $c = c' \circ \eta$. 
	
	This proves that $\Gha_0$ is isomorphic to $\Cay(G_0, \sS_1)$. To prove that $\Gha_1$ is isomorphic to $\Cay(G_1, \sS_1)$ we repeat the same argument multiplying on the left with a fixed element of $\sS_2$ throughout. Since $V(\Gha)$ is the disjoint union of $V_0$ and $V_1$, our statement follows.
\end{proof}

\begin{prop} \label{reg}
	For every \bpp\ $P = \langle \sS_1, \sU, \sI \vert \sR_0, \sR_1 \rangle$, the graph $\Gha := \Sp(P)$ is $\vert \sS \cup \sS^{-1} \vert$-regular. 
\end{prop}
\begin{proof}
	By \Prr{Cay G0}, the subgraph with edges coloured by $\sS_1 \cup \sS_1^{-1}$ is $2\vert \sS_1 \vert$-regular. It therefore suffices to prove that every vertex in $\Gha$ has a unique outgoing edge coloured $s$ for every $s \in \sS_2 \cup \sS_2^{-1}$. Existence is easy by the definition of $T_P$. To prove uniqueness, suppose in $T_P$ we have two edges $(v_0,u_0), (v_1,u_1) \in \overrightarrow{E}(T_P)$ where $c(v_0,u_0) = s = c(v_1,u_1)$ and $v_0 \sim v_1$. So by definition $u_i = v_i s$ and $v_0^{-1}v_1 \in R_i$ for $i \in \bZ/2\bZ$. Note that 
	\[
	u_0^{-1}u_1 = s^{-1} v_0^{-1} v_1 s =  s^{-1} (v_0^{-1}v_1) s \in s^{-1} R_i s \subset R_{i+1},
	\] 
	which means that $u_0 \sim u_1$ and hence $(v_0,u_0) \sim (v_1,u_1)$ proving our uniqueness statement.
\end{proof}

\begin{cor} \label{part}
	For a \bpp\ $P = \Spx$ the universal cover of $\Gha := \Sp(P)$ is $T_P$. Moreover,  every edge with a colour in $\sS_1$ connects two vertices in $V_i$ for some $i \in \bZ/2\bZ$,  and every edge with a colour in $\sS_2$ connects a vertex in $V_{i}$ to a vertex in $V_{i+1}$. 
\end{cor}
\begin{proof}
Recall that $\sim$  defines a map of graphs $\eta: T_P \rightarrow \Sp(P)$, by $\eta(x) = [x]$. As both $T_P$ and $T_P/\sim$ are $\vert \sS \cup \sS^{-1} \vert$-regular by \Prr{reg},  and $\eta$ is locally injective,  $\eta$ is a cover. As the fundamental group of a tree is trivial we deduce that $\eta$ is in fact the universal cover. 

By \Prr{Cay G0}, edges labelled $\sS_1$ connect vertices in $G_i$ to vertices in $G_i$, which are exactly the vertices in $V_i$. Moreover, in $T_P$ edges labelled $\sS_2$ connect vertices in $\tilde{V}_i$ to $\tilde{V}_{i+1}$. Therefore, edges labelled $\sS_2$ in $\Gha$ connect vertices in $\tilde{V}_i/\sim \ = V_i$ to vertices in $\tilde{V}_{i+1}/ \sim \ = V_{i+1}$.  
\end{proof}

\subsection{Topological definition}

We now give an alternative definition of $\Gha= \Sp(P)$ following the standard topological approach of defining a \Cg.

\medskip
Let $X$ be a set. Define the \defin{rose} $\Ro_{X}$ to be a graph with a single vertex $v$ and edge set $E(\Ro_{X}) = X$, where each $x \in X = E(\Ro_{X})$ signifies a loop at $v$. To be more precise, we let 
$X^{-1}$ denote an abstract set disjoint from $X$ and in bijection (denoted ${}^{-1}$) with $X$, and let $X\cup X^{-1}$ be the set of directed edges of $\Ro_{X}$. The terminus map $\tau$ of $\Ro_{X}$ maps all edges to $v$. We colour this rose by $c: \overrightarrow{E}(\Ro_{X}) \rightarrow {X} \cup {X}^{-1}$ by an arbitrary choice of orientation; in other words, $c$ is a bijection from $ \overrightarrow{E}(\Ro_{X})$ to $X\cup X^{-1}$ satisfying $c(e^{-1})= c(e)^{-1}$ for every $e\in X$.

For a presentation $P = \langle \sS \vert \sR \rangle$ of a group one often alternatively defines the Cayley graph in the following more topological way. 
We start by constructing the \defin{presentation complex} $\sC(P)$ as follows. The 1-skeleton of $\sC(P)$ is $\Ro_{\sS}$ with vertex $v$. 
For each relator $r  \in \sR$, we introduce a 2-cell $D_r$ and identify its boundary with the closed walk  of   $\Ro_{\sS}$ dictated by $r$ (see \Dr{def dict} below). 
This completes the definition of  $\sC(P)$. The Cayley graph $\Cay\langle \sS \vert \sR \rangle$ is the 1-skeleton of the universal cover of $\sC(P)$. 

We now generalise this construction to the context of our \bpp s. We remark that  it is not so easy to obtain the \invo\ Cayley graphs using this construction because $\Ro_{\sS}$ has even degree, so any cover will also have even degree. But treating $\sI$ appropriately  we will in fact be able to obtain graphs of odd degree. 

\begin{defn} \label{spectopdef}      
	Let $P = \langle \sS_1, \sU, \sI \vert \sR_0, \sR_1 \rangle$ be a \bpp. We construct the \defin{presentation complex} $\sC(P)$ of $P$ as follows. Start with two copies of $\Ro_{\sS_1}$, with vertices $v_0$ and $v_1$ respectively, and connect $v_0$ and $v_{1}$ with an edge for each element of $\sS_2 \cup \sS_2^{-1} \subset MF_P$. We will refer to this 1-complex $C(P)$ as the \defin{presentation graph} of $P$. We can extend the colouring of the two copies of $\Ro_{\sS_1}$ to a colouring $c: \overrightarrow{E}(C(P)) \rightarrow \sS \cup \sS^{-1}$ where $c(e)^{-1} = c(e^{-1})$. 
	
	To define the 2-cells of $\sC(P)$, 
for each relator $r = s_1s_2 \ldots s_n \in \sR_i$, we start a walk $p_r$ at $v_i$ and extend this walk inductively with the edge labelled $s_i, i=1, \ldots, n$. The path $p_r$ starts and ends at $v_i$ as $\sR_i \subset K$. Attach a 2-cell along each such closed walk $p_r$ to obtain the presentation complex $\sC(P)$ from $C(P)$. Finally, we define 
 the \defin{(topological) \bpcg} $\Sp'\langle\sS_1, \sU, \sI \vert \sR_0, \sR_1 \rangle$ to be the 1-skeleton of  the universal cover of $\sC(P)$.
\end{defn}

Our next result, \Tr{equiv-defn}, says that this gives rise to the same graph as in Definition \ref{specgrpdef}. To prove it, we will use the theory of covering spaces (\Sr{sec covers}). For this we need to turn our graphs into topological spaces, and we now recall the standard way to do so.

Given a graph $\Gha$ with vertex set $V$, and any orientation on its edges $O \subset \overrightarrow{E}(\Gha)$, we define a topological space as follows. Associate a point to each vertex, and a closed interval $I_e = [0,1]$ to each edge $e \in O$. Then define the quotient $I_e(0) \sim \tau(e^{-1})$ and $T_e(1) \sim \tau(e)$ to obtain the topological space
\[
\Gha = (V \cup \bigcup_{e \in O} I_e)/\sim.
\]
It is not hard to see that when $\Gha$ is  connected this topological space is path-connected, locally path-connected and semilocally simply-connected. Moreover, different choices of $O$ define homeomorphic topological spaces.

\medskip 
Next, we define a type of colouring that will be useful to establish that certain maps of graphs are covers.

\begin{defn} \label{def Cl}
	Let $\Gha$ be a graph with a colouring $c: \overrightarrow{E}(\Gha) \rightarrow X$. We say that $c$ is Cayley-like if
	\begin{enumerate}
		\item $\Gha$ is $\vert X \vert$-regular, 
		\item for all $e, e' \in \overrightarrow{E}(\Gha)$, if $c(e) = c(e')$ and $\tau(e) = \tau(e')$ then $e = e'$, and
		\item there is an involution $\,^{\textbf{-1}}: X \rightarrow X$ such that $c(e)^{\textbf{-1}} = c(e^{-1})$.
	\end{enumerate}
\end{defn}

Suppose we have two graphs $\Gha$ and $\Ghb$ with Cayley-like colourings $c_{\Gha}: \overrightarrow{E}(\Gha) \rightarrow X$ and $c_{\Ghb}: \overrightarrow{E}(\Ghb) \rightarrow X$. Then  any surjective map of graphs $\phi: \Gha \rightarrow \Ghb$ which respects these colourings, that is, $c_{\Gha} = c_{\Ghb} \circ \phi$, is a covering map of the associated topological spaces. Indeed, $\phi$ can't map any two edges that share an end vertex to the same edge, as this cannot respect the colourings.

Let $\sP_v(\Gha)$ be the set of walks in $\Gha$ starting at a vertex $v $, and define the group $MF_X$ by the presentation $\langle X \vert \{x x^{\textbf{-1}} : x\in X\} \rangle$. Then any Cayley-like colouring $c: \overrightarrow{E}(\Gha) \rightarrow X$ defines a map\\ $\sW_v: \sP_v(\Gha) \rightarrow  MF_X$ by $p = v e_1 v_1 \ldots e_n v_n \mapsto c(e_1) c(e_2) \ldots c(e_n)$. Note that there is a well defined inverse $\sW_v^{-1} : MF_X \rightarrow \sP_v(\Gha)$ as at every vertex $v' \in V(\Gha)$ there is a unique edge $e \in \overrightarrow{E}(\Gha)$ with colour $c(e)$ and $\tau(e^{-1}) = v'$. Moreover, $\sW_v^{-1}$ is a double sided inverse to $\sW_v$, so both these maps are bijections.  

\begin{defn} \label{def dict}
For any $g\in MF_X$, we say that  $\sW_v^{-1}(g) $ is the walk (in $\Gha$) \defin{dictated} by the word $g$ starting at $v$. 
\end{defn}
This is a natural definition since we can express $g$ as a word $s_1 \ldots s_n$ with $s_i\in X \cup X^{-1}$, and obtain $\sW_v^{-1}(g) $ by starting at $v$ and following the directed edges with colours $c(s_1)\ldots c(s_n)$; this is well-defined when $c$ is Cayley-like.

It is straightforward to check that if $p$ is homotopic to $p'$, then $\sW_v(p) = \sW_v(p')$. Thus by restricting  to the closed walks we can think of $\sW_v$ as a map from $\pi_1(\Gha,v)$ to $MF_X$, and so   the above remarks imply that 
\begin{prop} \label{Wisom}
$\sW_v$ is a group isomorphism from $\pi_1(\Gha,v)$ to a subgroup of  $MF_X$.
\end{prop}

Suppose we have a covering map of graphs $\psi: \Ghb \rightarrow \Gha$  both of which have Cayley-like colourings $c_{\Ghb}: \overrightarrow{E}(\Ghb) \rightarrow X$ and $c_{\Gha}: \overrightarrow{E}(\Gha) \rightarrow X$ such that $c_{\Ghb} = c_{\Gha} \circ \psi$. For a path $p: [0,1] \rightarrow \Gha$ with $p(0), p(1) \in V(\Gha)$ and a lift $\tilde{p}: [0,1] \rightarrow \Ghb$ of $p$ by $\psi$, it is straightforward to check that 
\labtequ{W cover}{$\sW_{p(0)}(p) = \sW_{\tilde{p}(0)}(\tilde{p})$}
where with a slight abuse, we interpreted $p$ as a walk in $\Gha$ in the obvious way. 

\begin{thm} \label{equiv-defn}
	For every \bpp\ $P = \langle \sS_1, \sU, \sI \vert \sR_0, \sR_1 \rangle$, the \bpcg s $\Gha = \Sp (P)$ and $\Ghb = \Sp\ ' (P)$ are isomorphic.
\end{thm}
\begin{proof}
	Our presentation graph $C=C(P)$ is $\vert \sS \cup \sS^{-1} \vert$-regular by definition. Therefore, the universal cover of $C$ is the $\vert \sS \cup \sS^{-1} \vert$-regular tree $T$, and we can let $\theta: T \rightarrow C$ be the corresponding covering map. Let $c_C: \overrightarrow{E}(C) \rightarrow \sS \cup \sS^{-1}$ be the colouring of $C$ as above. This lifts to a colouring $c_T : \overrightarrow{E}(T) \rightarrow \sS \cup \sS^{-1}$  of $T$, by letting $c_T(e) := c_C(\theta(e))$. This colouring allows us to identify $T$ with $T_P$. 
	
	Let $p \in \pi_1(C,v_i)$. As $c_C$ is a Cayley-like colouring of $C$, we can consider $\sW_{v_i}(p) \in MF_P$ by Definition~\ref{def Cl} and the discussion thereafter. Any closed walk representing $p$ must use an even number of edges coloured $\sS_2 \cup \sS_2^{-1}$ by the definition of $C$, so $\sW_{v_i}(p) \in K \subset MF_P$. Moreover, each $k \in K$ gives rise to a closed walk $\sW_{v_i}^{-1}(k)$ representing some element of  $\pi_1(C,v_i)$. Thus by Proposition~\ref{Wisom}, 
	\labtequ{pi C onto K}{$\sW_{v_i}$ is an isomorphism from $\pi_1(C,v_i)$ onto $K$. }
	
	Recall that we can identify $T$ with $T_P$. If in doing so we identify the identity $1_{MF_P} \in V(T_P)$ of $MF_P$ with some vertex in $\theta^{-1}(v_0)$ (which we easily can) then \eqref{pi C onto K} implies 
	\labtequ{thet Vi}{$\theta(\tilde{V}_i) = v_i$,}
	because $\tilde{V}_0 = K$ and $\tilde{V}_1 = \sS_2 K$.
	\showFigTikz{
		\begin{displaymath}
		\xymatrix{
			& T_P \ar[dl]^{\eta} \ar[dr]_{\Phi} \ar[drr]^{\widehat{\Phi}} \ar[dd]^{\theta}\\
			\Gha \ar[dr]^{\nu} & & \Ghb \ar[dl]^{\epsilon} \ar@{^{(}->}[r]_{i} & \widehat{\Ghb} \ar[dl]^{\widehat{\epsilon}}\\
			& C \ar@{^{(}->}[r]^i & \sC
		} \hspace{0.5 in}
		\xymatrix{
			& 1 \ar@{^{(}->}[dl]^{\eta_{\ast}} \ar@{^{(}->}[dr]_{\Phi_{\ast}} \ar@{^{(}->}[drr]^{\widehat{\Phi}_{\ast}} \ar@{^{(}->}[dd]^{\theta_{\ast}}\\
			\pi_1(\Gha) \ar@{^{(}->}[dr]^{\nu_{\ast}} & & \pi_1(\Ghb) \ar@{^{(}->}[dl]^{\epsilon_{\ast}} \ar@{->>}[r]_{i_{\ast}} & 1 \ar@{^{(}->}[dl]^{\widehat{\epsilon}_{\ast}}\\
			& K \ar@{->>}[r]^{i_{\ast}} & K/R_i
		}
		\end{displaymath}
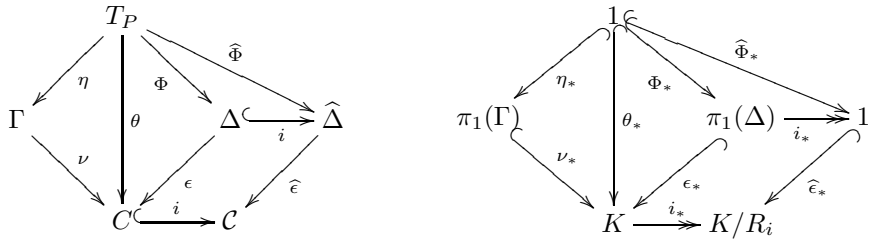
	}{maps-equiv-defn}{Maps used in Proposition \ref{equiv-defn}}
	
	Let $\eta : T_P \rightarrow \Gha$ be the covering map found in Corollary \ref{part}. Let $c_{\Gha} : \overrightarrow{E}(\Gha) \rightarrow \sS \cup \sS^{-1}$ be the colouring of $\Gha$ as in its definition. Now define a map $\nu : \Gha \rightarrow C$ by letting $\nu(v) = v_i$ whenever $v \in V_i = \eta(\tilde{V}_i)$. If $c_{\Gha}(e) = s$ for some $e \in \overrightarrow{E}(\Gha)$ then $\nu$ maps $e$ to the unique edge $e' \in \overrightarrow{E}(C)$ with $c_{C}(e') = s$ and $\tau(e') = \nu(\tau(e))$. Since for every $v \in \tilde{V}_i$ we have $\eta(v) \in V_i$, we have $\nu(\eta(v)) = v_i$ and hence $\theta = \nu \circ \eta$ by \eqref{thet Vi}.
	
	Let $\widehat{\epsilon} : \widehat{\Ghb} \rightarrow \sC$ be the universal cover of $\sC:= \sC(P)$. We know that $\Ghb$ and $C$ are the 1-skeletons of $\widehat{\Ghb}$ and $\sC$ respectively, so we obtain the inclusion maps $i: \Ghb \rightarrow \widehat{\Ghb}$ and $i: C \rightarrow \sC$. Furthermore, by restricting $\widehat{\epsilon}$ to the 1-skeleton we obtain a covering map $\epsilon: \Ghb \rightarrow C$. As $\theta: T_P \rightarrow C$ is the universal cover of $C$, it can be lifted through $\epsilon: \Ghb \rightarrow C$ to a map $\Phi: T_P \rightarrow \Ghb$ so that $\epsilon \circ \Phi = \theta$ by the definition of a universal cover. This gives us a map $\widehat{\Phi}: T_P \rightarrow \widehat{\Ghb}$ defined by $\widehat{\Phi}:= i \circ \Phi$. Note that all these maps respect the colourings of the edges as $\theta$ and $\widehat{\epsilon}$ do.
	
	By Theorem \ref{cover-correspondence}, to show $\Gha \cong \Ghb$ it suffices to show that $\nu_{\ast}(\pi_1(\Gha)) = \epsilon_{\ast}(\pi_1(\Ghb))$, or equivalently $\sW_{v_i}(\nu_{\ast}(\pi_1(\Gha))) = \sW_{v_i}(\epsilon_{\ast}(\pi_1(\Ghb)))$ as $\sW_{v_i}$ is a bijection. To do so, we will prove that the latter groups are both equal to $R_i$, where $R_i$ is as defined after  Definition~\ref{spec sp pres def}. 
	
	To show that $\sW_{v_i}(\nu_{\ast}(\pi_1(\Gha))) = R_i$, let $p$ be a closed walk representing some element of  $\pi_1(\Gha,v)$ with $v \in V_i$. Choose a lift of $p$ to a walk $\tilde{p}: [0,1] \rightarrow T_P$ (so $\eta \circ \tilde{p} = p$). We know that $\eta(\tilde{p}(0)) = \eta(\tilde{p}(1)) = v$, so $\tilde{p}(0),\tilde{p}(1) \in \eta^{-1}(v)$ implying $\tilde{p}(0)^{-1}\tilde{p}(1) \in R_i$. So $\sW_{v_i}(\nu_{\ast}(p)) = \sW_{v_i}(\theta(\tilde{p})) \in R_i$, which proves that  $\sW_{v_i}(\nu_{\ast}(\pi_1(\Gha))) \subseteq R_i$
	
	We would like to use Proposition~\ref{Wisom} to deduce $\sW_{v_i}(\nu_{\ast}(\pi_1(\Gha))) = R_i$, and for this it now only remains to prove that the former is surjective onto $R_i$. To show this, pick any $r \in R_i$.  As $R_i \subset K \cong \sW_{v_i}(\pi_1(C,v_i))$ by \eqref{pi C onto K}, there is a representative $q$ of an element of $\pi_1(C,v_i)$ such that $\sW_{v_i}(q) = r$. Choose a lift $\tilde{q} : [0,1] \rightarrow T_P$ of $q$ through $\nu \circ \eta = \theta$, such that $\eta(\tilde{q}(0))=v$ (and so $\nu \circ \eta \circ \tilde{q} = \theta \circ \tilde{q} = q$). Then as $\sW_{v}(\tilde{q}) = \sW_{v_0}(q) = r \in R_i$ we have  $\tilde{q}(0)^{-1}\tilde{q}(1) \in R_i$, and so $\tilde{q}(0) \sim \tilde{q}(1)$, with $\sim$  as in the definition of $\Gha$ as a  quotient of $T_P$. This means that $\eta(\tilde{q}(1)) = \eta(\tilde{q}(0)) = v$, and so $\eta \circ \tilde{q}$ is a loop representing an element of $ \pi_1(\Gha,v)$. Since $\nu_{\ast}(\eta \circ \tilde{q}) = \theta \circ \tilde{q} = q$ represents an element of  $\nu_{\ast}(\pi_1(\Gha))$ we deduce that $r = \sW_{v_i}(q) \in \sW_{v_i}(\nu_{\ast}(\pi_1(\Gha,v)))$, proving that $ \sW_{v_i}(\nu_{\ast}(\pi_1(\Gha,v)))$ surjects onto $R_i$ as desired. 

	Next, we prove $\sW_{v}(\epsilon_{\ast}(\pi_1(\Ghb,v))) \subseteq R_i$ for every $v \in V(\Ghb)$ with $\epsilon(v) = v_i$. It is well-known  \cite[Proposition 1.26]{HA02} that the inclusion of the one skeleton into a 2-simplex induces a surjection on the level of fundamental groups, and the kernel is exactly the normal closure of the words bounding the 2-cells. Thus $i_{\ast} : \pi_1(C,v_i) \rightarrow \pi_1(\sC,v_i)$ is a surjection. Combining these remarks with \eqref{pi C onto K}, it follows that $i_{\ast} \circ \sW_{v_i}^{-1}: K \rightarrow \pi_1(\sC,v_i)$ is a surjection, with kernel $R_i$, since $R_i$ is the normal closure in $K$ of the words onto which $\sW_{v_i}^{-1}$ maps the closed walks bounding 2-cells of $\sC$ by the definition of $\sC$. Thus $\pi_1(\sC,v_i) = K/R_i = G_i$. Now pick $v \in V(\Ghb)$ with $\epsilon(v) = v_i$. As $i \circ \epsilon = \widehat{\epsilon} \circ i$ and $\pi_1(\widehat{\Ghb}) = 1$, we have $(i_{\ast} \circ \epsilon_{\ast})(\pi_1(\Ghb,v)) = (\widehat{\epsilon}_{\ast} \circ i_{\ast}) (\pi_1(\Ghb,v)) = 1$, and so $\sW_{v}(\epsilon_{\ast}(\pi_1(\Ghb,v))) \leq ker(i_{\ast}) = R_i$ as desired. 

	Finally, we claim that $R_i \subset \sW_{v}(\epsilon_{\ast}(\pi_1(\Ghb,v)))$  for every $v \in V(\Ghb)$ with $\epsilon(v) = v_i$. For this, pick $r \in R_i$, and note that as $R_i \subset K$ and $K \cong \sW_{v_i}(\pi_1(C,v_i))$ by \eqref{pi C onto K}, there is a representative $t$ of an element of $\pi_1(C,v_i)$ such that $\sW_{v_i}(t) = r$. We can write $\sW_{v_i}(t) = r = \prod_{j=1}^n w_jr_jw_j^{-1} \in MF_P$ for $w_j \in K$ and $r_j \in \sR_i \cup s\sR_{i+1}s^{-1}$ with $s \in \sS_2 \cup \sS_2^{-1}$ by the definition of $R_i$. Choose a lift $t' : [0,1] \rightarrow \Ghb $ of $t$ through $\epsilon$ so that $t'(0) = v$. By \eqref{W cover} we have $\sW_{v_i}(t) = \sW_{v}(t')$. Note that $\sW^{-1}_{v}(w_jr_jw_j^{-1})$  is a loop of $\Ghb$ as $\sW^{-1}(r_j)$ is contractable in $\widehat{\Ghb}$, and so it represents some element of $\pi_1(\Ghb,v)$. Applying this to each factor of our above expression $r = \prod_{j=1}^n w_jr_jw_j^{-1}$ implies that $t'$ represents some element of  $\pi_1(\Ghb,v)$. Thus $\sW_{v_i}(\epsilon_{\ast}(t')) = \sW_{v_i}(t) = r$, which means that $R_i \subset \sW_{v}(\epsilon_{\ast}(\pi_1(\Ghb,v)))$ as claimed.

	
	To summarize, we have proved that $\sW_{v_i}(\nu_{\ast}(\pi_1(\Gha))) = R_i = \sW_{v_i}(\epsilon_{\ast}(\pi_1(\Ghb)))$, implying that $\Gha \cong \Ghb$. Moreover, it is straightforward to check that as all the maps above respect the edge colourings, so does this isomorphisms of graphs.
\end{proof}

From now on we just use the notation $\Sp(P)$ for the \bpcg\ obtained in either Definition~\ref{specgrpdef} or~\ref{spectopdef}.

As a corollary of the above proof, we deduce that the covers $\nu,\epsilon$ are equal, and so 
\labtequ{eps vV}{$V_i = \nu^{-1}(v_i) = \epsilon^{-1}(v_i)$}
 and similarly $V_i = \eta(\tilde{V}_i) = \Phi(\tilde{V}_i)$, so $V_i$ is well defined for either the topological or graph definition, as in the notation of Figure \ref{maps-equiv-defn}. From now on we will only use $\epsilon$ to denote this covering map. \mymargin{remove or reinforce}

The following corollary gathers some further facts that we obtained in the proof of Theorem~\ref{equiv-defn} for future reference.

\begin{cor} \label{pi-one}
	Let $P = \langle \sS_1, \sU, \sI \vert \sR_0, \sR_1 \rangle$ be a \bpp\ with \pcg\ $\Gha := \Sp(P)$. We have
	\begin{enumerate}
		\item \label{pi_one I} $\pi_1(\sC(P),v_i)$ is isomorphic to  $G_i$;
		\item $\sW_{v_i}$ is an isomorphism from $\pi_1(C(P), v_i)$ onto $K$;
		\item $\sW_{v_i}$ is an isomorphism from  $\pi_1(\Gha,v)$ onto $R_i$ for every $v \in V_{i}$; and
		\item the sequence $0 \rightarrow \pi_1(\Gha,v) \xrightarrow{\epsilon_{\ast}} \pi_1(C(P), v_i) \xrightarrow{i_{\ast}} \pi_1(\sC(P), v_i) \rightarrow 0$ is exact, where $\epsilon : \Gha \rightarrow C(P)$ is the cover in Definition~\ref{spectopdef}, and $i: C(P) \rightarrow \sC(P)$ the inclusion.  
	\end{enumerate}
\end{cor}

\bigskip
The \defin{generalised Petersen graph} is denoted by $P(n,k)$ and defined as follows. Let 
\begin{align*}
V(P(n,k)) &:= \{x_i, y_i \ \vert \ i \in \bZ /n\bZ \}, \mbox{ and} \\
E(P(n,k)) &:= \{ (x_i, x_{i+1}), (x_i, y_i), (y_i, y_{i+k}) \ \vert \ i \in \bZ /n \bZ \}.
\end{align*}
The classical example is the Petersen graph, $P(5,2)$, the smallest non-Cayley vertex transitive graph. The following statement, proved in the second author's PhD thesis, says that we can obtain every $P(n,k)$ as a \bpcg.

\begin{thm} \label{Petersen-Split-Presentation}
	The generalised Petersen graph $P(n,k)$ is isomorphic to $\Sp\langle \{a\}, \emptyset, \{b\} \vert \{a^n,aba^kb\},\{a^n\} \rangle$. 
\end{thm}

Note that from the definition of $R_i$ we have $R_0 = sR_{1}s^{-1}$ for any $s \in \sS_2$.\mymargin{move} Therefore, we deduce that $G_i := R_i \backslash K \cong R_{i + 1} \backslash K$, where an isomorphism $\phi_{s,i} : G_i \rightarrow G_{i + 1}$ is given by conjugation by any $s \in \sS_2$. This property isn't enough to guarantee vertex transitivity of $\Gha$, with a counter example given by $P(4,2)$. This invites the following rather inconsice question. 

\begin{que} \label{Q vt}
	For which \bpp s $P$ is $\Sp (P)$ vertex transitive? 
\end{que}

\section{Relationships to Bi-Cayley and Haar graphs} \label{sec Bi-Cayley}

We recall that an action on a graph $\Gha$ is \defin{semi-regular} (or \defin{free}) if $g \cdot x = h \cdot x$ implies $g = h$ for every $g,h \in \Gpa$ and $x\in V(\Gha)$.
A vertex transitive graph $\Gha$ is said to be \defin{$n$-Cayley over $\Gpa$} if \Gpa\ is a semi-regular subgroup of $\Aut(\Gha)$ with $n$ orbits of vertices. If $n = 2$ we say that $\Gha$ is \defin{bi-Cayley}.

Suppose $\Gha$ is {bi-Cayley} over \Gpa. Pick two vertices $e_0, e_1 \in V(\Gha)$ from different orbits of $\Gpa$. As $\Gpa$ has exactly two orbits in $V(\Gha)$, and it acts regularly on each of them, for any $x \in V(\Gha)$ there exists a unique $i \in \bZ/2\bZ$ and $g \in G$ such that $g \cdot e_i = x$, so we define $x =: (g)_i$. Each of the two orbits $O_i:= \{(g)_i : g \in \Gpa\}$ forms a (possibly disconnected) Cayley graph of $\Gpa$ with respect to the generating sets $R = R^{-1} = \{g \in \Gpa \vert e_0 (g)_0 \in E(\Gha)\}$ and $L = L^{-1} = \{g \in \Gpa \vert e_1 (g)_1 \in E(\Gha)\}$, respectively. Here, by $vw\in E(\Gha) $ we mean that either $(v,w) \in \overrightarrow{E}(\Gha)$ or  $(w,v) \in \overrightarrow{E}(\Gha)$. To capture the set $E_{01}$ of edges of the form $(g)_0 (h)_1 \in E(\Gha)$, we introduce the set $S = \{g \in \Gpa \vert (e_0, (g)_1) \in \overrightarrow{E}(\Gha)\}$ and note that $S$ uniquely determines $E_{01}$ as any $e \in E_{01}$ coincides with $h \cdot e_0 (g)_1$ for some $g \in S$ and $h\in \Gpa$.

To summarize, we can represent any bi-Cayley graph $\Gha$ over \Gpa\  as $\Bi(\Gpa,R,L,S)$ where $R, L, S \subset \Gpa$ with $R = R^{-1}$ and $L = L^{-1}$. Then the set of directed edges of $\Gha =: \Bi(\Gpa,R,L,S)$ is 
\begin{align*}
\{((g)_0, (gr)_0) \vert  g \in \Gpa, r \in R\} \cup \{ ((g)_1, (gl)_1) \vert  g \in \Gpa, l \in L\}\\ 
\cup \{((g)_{0}, (gs)_{1}) \vert g \in \Gpa, s \in S\} \cup \{((g)_1, (gs^{-1})_0) \vert g \in \Gpa, s \in S\}.
\end{align*}
 This representation isn't unique: if we choose different vertices for $e_0, e_1$ or a different action of $\Gpa$ we potentially obtain different sets $R$, $S$ and $L$. Note that  $\Bi(\Gpa,R,L,S)$ is regular if and only if $\vert R \vert = \vert L \vert$.

\begin{examp} \label{PetersenFirstExamp} 
	Consider again the Petersen graph $\Gha = P(5,2)$ as in Example~\ref{ex Pet 1} (Figure~\ref{Petersen-no-colour}).	This has a natural action of $G:= \bZ/5\bZ = <a>$ where
	\[
	a^j : \begin{array}{c}
	x_i\\y_i
	\end{array} \mapsto \begin{array}{c}
	x_{i+j}\\y_{i+j}
	\end{array}.
	\] 
	To represent this as a bi-Cayley graph with above notation, we could choose $(a^0)_0 := x_0$ and $(a^0)_1 := y_0$. Then we obtain $R = \{a,a^4\}$, $L = \{a^2,a^3\}$ and $S = \{a^0\}$. If instead we chose $(a^0)_1 := y_1$ we would obtain $R = \{a,a^4\}$, $L = \{a^2,a^3\}$ and $S = \{a^4\}$.
\end{examp}
	\showFigTikz{
		\begin{tikzpicture}[scale = 0.7]
		\node (G-11) at (-2,0) [circle,fill=black,scale = 0.3, label=left:{$x_2$}] {};
		\node (G-12) at (2,0) [circle,fill=black,scale = 0.3, label=right:{$x_3$}] {};
		\node (G-13) at (3,3) [circle,fill=black,scale = 0.3, label=right:{$x_4$}] {};
		\node (G-14) at (0,5) [circle,fill=black,scale = 0.3, label=above:{$x_0$}] {};
		\node (G-15) at (-3,3) [circle,fill=black,scale = 0.3, label=left:{$x_1$}] {};
		\node (G-21) at (-1,1) [circle,fill=black,scale = 0.3,label=below:{$y_2$}] {};
		\node (G-22) at (1,1) [circle,fill=black,scale = 0.3,label=below:{$y_3$}] {};
		\node (G-23) at (2,3) [circle,fill=black,scale = 0.3,label=below:{$y_4$}] {};
		\node (G-24) at (0,4) [circle,fill=black,scale = 0.3, label=right:{$y_0$}] {};
		\node (G-25) at (-2,3) [circle,fill=black,scale = 0.3,label=below:{$y_1$}] {};
		\draw (G-11) edge (G-12);
		\draw (G-12) edge (G-13);
		\draw (G-13) edge (G-14);
		\draw (G-14) edge (G-15);
		\draw (G-15) edge (G-11);
		\draw (G-21) edge (G-23);
		\draw (G-23) edge (G-25);
		\draw (G-25) edge (G-22);
		\draw (G-22) edge (G-24);
		\draw (G-24) edge (G-21);
		\draw (G-11) edge (G-21);
		\draw (G-12) edge (G-22);
		\draw (G-13) edge (G-23);
		\draw (G-14) edge (G-24);
		\draw (G-15) edge (G-25);
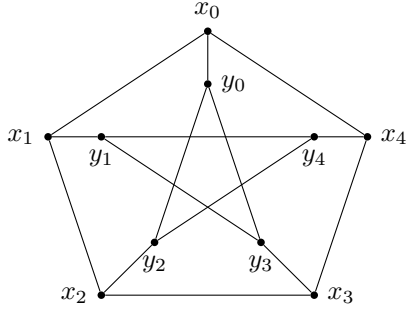
		\end{tikzpicture}}{Petersen-no-colour}{The labelling of the Petersen graph used in Example~\ref{PetersenFirstExamp}.}

Recall that we have endowed $\Gha := \Spx$ with a colouring $c: \overrightarrow{E}(\Gha) \rightarrow \sS \cup \sS^{-1}$. We want to talk about automorphisms that preserve this colouring. The following definition distinguishes between preserving these colours globally or locally.

\begin{defn}
	Let $\Gha$ be a graph with a colouring $c: \overrightarrow{E}(\Gha) \rightarrow X$. We define the following two subgroups of $Aut(\Gha) $:
	\begin{align*}
	Aut_c(\Gha) =&  \{\phi \in Aut(\Gha) \vert c(e) = c(\phi(e)) \text{ for every } e\in \overrightarrow{E}(\Gha)\}, \mbox{ and}\\ Aut_{c-loc}(\Gha) = & \{ \phi \in Aut(\Gha) \vert c(x,y) = c(y,z) \Leftrightarrow c(\phi(x,y)) = c(\phi(y,z)) \mbox{ for all } (x,y),(y,z) \in \overrightarrow{E}(\Gha)\}. 
	\end{align*} 
\end{defn}

We remark that for any \bpp\ $P$, there is a subgroup of $Aut_c(\Sp(P))$ witnessing that $\Sp(P)$ is a bi-Cayley graph:

\begin{prop} \label{verttransonV}
	For every \bpp\ $P = \langle \sS_1, \sU, \sI \vert \sR_0, \sR_1 \rangle$ the vertex group $G_i$ is a subgroup of $\Aut_c(\Sp(P))$. Moreover $G_i$ acts regularly on $V_i$ (and on $V_{i+1}$) for $i \in \bZ/2\bZ$, and so $\Sp(P)$ is bi-Cayley over $G_0 \cong G_1$.
\end{prop}

\begin{proof} 
	Recall that for a covering map $\eta: X \rightarrow Y$, the group of automorphisms $f: X \rightarrow X$ such that $\eta \circ f = \eta$ is called the \defin{deck group} of $\eta$ and is denoted by $\Aut(\eta)$. It is known that if $\eta$ is a universal cover $\Aut(\eta) = \pi_1(Y)$, and if $X$ is connected and locally path connected then $\Aut(\eta)$ acts freely on $\eta^{-1}(y)$ for any $y \in Y$ \cite{HA02}.
	
	Let $\Gha := \Sp(P)$ and let $\widehat{\eta}: \widehat{\Gha} \rightarrow \sC$ be the universal cover of the presentation complex $\sC(P)$ of $P$. Thus $\Aut(\widehat{\eta}) \cong \pi_1(\sC(P)) \cong G_i$ by the above remark and Corrolary~\ref{pi-one}~\eqref{pi_one I}. As $\Gha$ is the 1-skeleton of $\widehat{\Gha}$ by Definition~\ref{spectopdef}, we can think of $\Aut(\widehat{\eta}) \cong G_i$ as a subgroup of $ \Aut(\Gha)$. Moreover as elements of $\Aut(\widehat{\eta}) \cong G_i$ preserve the cover, they preserve the colouring $c: \overrightarrow{E}(\Gha) \rightarrow \sS \cup \sS^{-1}$ obtained by lifting our colouring of $\sC(P)$ via $\widehat{\eta}$, and so we have realised $G_i$ as a subgroup of $\Aut_c(\Gha)$. As $\sC(P)$ is a connected 2-complex it is locally path connected, therefore $G_i$ acts freely on $\widehat{\eta}^{-1}(v_i) = V_i$ by the above remarks.
\end{proof}

\begin{prop}
	Every regular connected bi-Cayley $\Bi(\Gpa,R,L,S)$ graph where $R \cap R^{-1} =  L \cap L^{-1} = \emptyset$ and $\vert R \vert = \vert L \vert$ can be constructed as a \bpcg.
\end{prop}
\begin{proof}
	Let $\Gha := \Bi(\Gpa,R,L,S)$ be a a bi-Cayley graph, and recall our representation of its vertex set as $V(\Gha) = \{(g)_i \vert g \in \Gpa, \ i \in \bZ/2\bZ\}$. Choose $\sS_1 \subset R$ such that $\sS_1 \cap \sS_1^{-1} = \emptyset$ and yet $\sS_1 \cup \sS_1^{-1} = R$. Choose a bijection $f: L \rightarrow R$ such that $f(s^{-1}) = f(s)^{-1}$. We use $f$ to define the colouring $c: \overrightarrow{E}(\Gha) \rightarrow \sS_1 \cup \sS_1^{-1} \cup S$ as follows:
	\[
	c : \begin{array}{c}
	((g)_0, (rg)_0)\\
	((g)_1, (lg)_1)\\
	((g)_0, (sg)_{1})^{\pm 1}
	\end{array} \mapsto \begin{array}{c}
	r\\
	f(l)\\
	s
	\end{array} \ \mbox{ for } \ \begin{array}{c}
	r \in R\\
	l \in L\\
	s \in S.
	\end{array}
	\]
	Note that this colouring is Cayley-like, as there is a unique edge of each colour incident with each vertex. Let $\sI := S$, and set $\sR_0 := \sW_{(1_{\Gpa})_0}(\pi_1(\Gha,(1_{\Gpa})_0))$. We have thus constructed a \bpp\ $P := \langle \sS_1, \emptyset, \sI \vert \sR_0, \emptyset \rangle$. We claim that $\Gha \cong \Sp(P)$. 
	
	To see this, let as usual $\sC(P) =: \sC$ be the presentation complex and $C(P) =: C$ the presentation graph with vertices $v_i, i \in \bZ/2\bZ$ and edges $\overrightarrow{E}(C) = \{r(v_i,v_i), s(v_i, v_{i+1}) \vert r \in \sS_1 \cup \sS_1^{-1}, \ s \in S = \sI\}$ where $c_C(x(v_i,v_j)) = x$. We will prove $\Gha \cong \Sp(P)$ by applying Theorem \ref{cover-correspondence} to a cover  $\eta : \Gha \rightarrow C$ defined by $\eta: (g)_i \mapsto v_i$, and
	\[
	\eta : \begin{array}{c}
	((g)_0, (rg)_0)\\
	((g)_1, (lg)_1)\\
	((g)_0, (sg)_{1})^{\pm 1}
	\end{array} \mapsto \begin{array}{c}
	r(v_0,v_0)\\
	f(l)(v_1,v_1)\\
	(s(v_0,v_1))^{\pm 1}
	\end{array} \ \mbox{ for } \ \begin{array}{c}
	r \in R\\
	l \in L\\
	s \in S
	\end{array}.
	\]
	
	As $\eta$ is a map of graphs with Cayley-like colourings, and $\eta$ respects these colourings by definition, it is indeed a cover. We have  $\sW_{v_0}(\eta_{\ast}(\pi_1(\Gha,(1_{\Gpa})_0))) = \sR_0$ by  the choice of $\sR_0$. Let $\epsilon: \Sp(P) \rightarrow C$ represent the cover given in definition~\ref{spectopdef} of $\Sp(P)$ (as in Figure~\ref{maps-equiv-defn}). By Corollary \ref{pi-one} (3) we have that $\sW_{v_0}(\epsilon_{\ast}(\pi_1(\Sp(P)),v)) = R_0 := \langle \langle \sR_0 \rangle \rangle_K$ for some $v \in V(\Sp(P))$ such that $\epsilon(v) = v_0$. Note that for any $k \in K$ the path $\sW_{(1_{\Gpa})_0}^{-1}(k)$ connects $(1_{\Gpa})_0$ to $(g)_0$ for some $g \in \Gpa$ because it uses an even number of  edges $e$ with $c(e)\in S$. This implies 
\labtequ{gG}{$\sW_{(1_{\Gpa})_0}^{-1}(k)\pi_1(\Gha,(g)_0) \sW_{(1_{\Gpa})_0}^{-1}(k)^{-1} = \pi_1(\Gha,(1_{\Gpa})_0)$.}
As there exists a colour preserving automorphism of $\Gha$ mapping $(1_{\Gpa})_0$ to $(g)_0$, namely $g$, we moreover have  
\labtequ{WW}{$\sW_{(g)_0}(\pi_1(\Gha,(g)_0)) = \sW_{(1_{\Gpa})_0}(\pi_1(\Gha,(1_{\Gpa})_0))$.}
Therefore  $\langle \langle \sR_0 \rangle \rangle_K = \sR_0$ by \eqref{gG}, \eqref{WW} and the definition of $\sR_0$. Using this we have $\sW_{v_0}(\epsilon_{\ast}(\pi_1(\Sp(P),v))) = \sR_0$.
Moreover, we have $\sW_{v_0}(\eta_{\ast}(\pi_1(\Gha,(1_{\Gpa})_0))) = \sR_0$ by the definition of $\sR_0$ and \eqref{W cover}. Therefore $\sW_{v_0}(\epsilon_{\ast}(\pi_1(\Sp(P),v))) = \sW_{v_0}(\eta_{\ast}(\pi_1(\Gha,(1_{\Gpa})_0)))$, and so by Theorem \ref{cover-correspondence} we have $\Gha \cong \Sp(P)$.
\end{proof}

A \defin{Haar graph} is a bi-Cayley graph of the form $\Bi(\Gpa,\emptyset,\emptyset,S)$. The following is an immediate consequence of the last two propositions.

\begin{cor}
	Every Haar graph can be represented as a \bpcg, and every \bpcg\ $\Sp(\langle \sS_1, \sU, \sI \vert \sR_0, \sR_1 \rangle)$ with $\sS_1 = \sU = \emptyset$ is a Haar graph.
\end{cor}

Most of our motivation for introducing  \pp s came from studying vertex transitive graphs. Our next  proposition gives a sufficient condition for $\Sp(P)$ to be vertex transitive in terms of the `symmetry' of $\sC(P)$. Given two CW complexes $\sC_i$ for $i \in \bZ/2\bZ$, recall that a \defin{simplicial map} $\phi : \sC_0 \rightarrow \sC_1$ is a continuous map that maps each $n$-simplex to an $n$-simplex for every $n$. For a CW complex $\sC$, the group of bijective simplicial maps from $\sC$ to itself is denoted by $\Aut(\sC)$.

\begin{prop} \label{prestrans}
	Let $P$ be a \bpp. As above, the two vertices of the presentation complex $\sC$ are denoted by $v_0$ and $v_1$. If there exists a simplicial map $\phi: \sC \rightarrow \sC$ such that $\phi(v_0) = v_1$, then $\Sp(P)$ is vertex transitive. 
\end{prop}
\begin{proof}
Set $\Gha := \Sp(P)$. Lemma \ref{verttransonV} says that $G_i$ acts transitively on $V_j$ for $j \in \bZ/2\bZ$. Thus it only remains to find an automorphism which maps a vertex in $V_0$ to a vertex in $V_1$. We have a covering map $\epsilon: \widehat{\Gha} \rightarrow \sC$, where $\widehat{\Gha}$ is the universal cover of $\sC$ with 1-skeleton $\Gha$. By the lifting property $\phi \circ \epsilon : \widehat{\Gha} \rightarrow \sC$ lifts to an automorphism $\widehat{\phi} \in \Aut(\widehat{\Gha})$ such that $\phi \circ \epsilon = \epsilon \circ \widehat{\phi}$. For any $v \in V_{i}$ we have  $\epsilon(v) = v_i$ by \eqref{eps vV}. So for $v \in V_0$ we have $\epsilon \circ \widehat{\phi}(v) = \phi \circ \epsilon(v) = \phi(v_0) = v_{1}$ giving that $\widehat{\phi}(v) \in V_{1}$. Thus when restricting $\widehat{\phi}$ to the 1-skeleton, $\Gha$, we obtain the required automorphism.
\end{proof}

We remark that this sufficient condition is not necessary for $\Sp(P)$ to be vertex transitive. For example,   there is never such an automorphism for the \pp s  $\langle \{a\}, \{\}, \{b\} \vert \{a^n,aba^kb\},\{a^n\} \rangle$ of \Tr{Petersen-Split-Presentation} unless $k = \pm 1$. However, we know that $P(n,k)$ is transitive for many other choices of $n$ and $k$ (such as the case of the Petersen graph $n = 5,k = 2$), 
see \cite{FGW71}.

\section{$n$-\pp s for $n>2$} \label{sec general}

\comment{
Before stating a definition we want to talk in abstract about what a `presentation' so far is. A presentation consists of the following data
\begin{itemize}
	\item a regular graph $\Gha$,
	\item a \pf\ weak multicycle colouring of $c: \overrightarrow{E}(\Gha) \rightarrow \sS \cup \sS^{-1}$, and
	\item a set of relators for each $v \in V(\Gha)$ with $\sR_v \subset \pi_1(\Gha,v)$.
\end{itemize}
For a group presentations $\langle \sS \vert \sR \rangle$ the graph is $Ro_{\sS}$ with each edge coloured by $\sS$. Then using $c$ we can identify $\pi_1(Ro_{\sS},v) = F_{\sS}$, so the relator set $\sR \subset F_{\sS} = \pi_1(Ro_{\sS},v)$.

For a \bpp\ $P = \langle \sS_1, \sU, \sI \vert \sR_0, \sR_1 \rangle$ the graph is $C(P)$ with the colouring $c$ given by the construction of $\sC(P)$. Then using $c$ we can identify $\pi_1(C(P), v_i) = K \subset MF_P$, therefore $\sR_i \subset K$ gives our relator sets.

Note that to have a `presentation' in our sense, the regular graph $\Gha$ needs to have a weak multicycle colouring. A simple corollary of Theorem \ref{2-regular-sub-graph} says all even degree regular graphs have such a colouring. However by Tutte's theorem \cite{BM76} one can construct $2k+1$-regular graphs with no \pf\ multicycle colouring see Example \ref{no-weak-multicycle-colouring}.
}
\subsection{Definition of \pp s} \label{sec def gsp}

In this section we generalise our notion of \pp\ by allowing for more than two classes of vertices $V_i$. This will allow us to describe vertex transitive graphs such as the Coxeter graph 
which cannot be expressed as a bi-Cayley graph. 

In \Dr{spec sp pres def} of a \bpp\ we did not explicitly talk about the two vertex classes, but they were implicit in that definition: we had two sets of relators $\sR_0, \sR_1$, and the definition of $K$ implicitly distinguished our generators into those staying in the same vertex class, namely $\sS_1$, from those swapping between the two vertex classes, namely  $\sS_2$. The two vertex classes $V_i$ were defined a-posteriori, and \Cr{part} confirms that the generators gave rise to edges of the \pcg\ behaving this way.

The following definition is a direct generalisation of \Dr{spec sp pres def}, although it is formulated a bit differently. We now make the vertex classes more explicit. The main complication arises from the fact that we have to specify, for each generator $s$, which vertex class any edge coloured by $s$ will lead into if it starts at a given vertex class. This information is encoded as a permutation $\phi(s)$ of the set of vertex classes. As before, we distinguish our generators into two subsets $\sU$ and $\sI$ to allow for `involutions' that make \pcg s with odd degrees possible.

We now give the formal definition:

\begin{defn} \label{def GSP}
	A \pp\ $\langle X \vert \sU \vert \sI \vert \phi \vert \sR \rangle$ consists of the following data:
	\begin{enumerate}
		\item
		a set of vertex classes $X$;
		\item \label{SP ii}
		a generator set $\sS$, which is partitioned into two sets $\sU$ and $\sI$; as before, we use $\sS$ to define a group $MF_P := \langle \sS \vert \{s^2 \vert s \in \sI \} \rangle$ (a free product of cyclic groups each of order 2 or $\infty$);
		\item \label{SP iii}
		a map $\phi: \sS \rightarrow \Sym_{X}$ from the generator set to the group $\Sym_{X}$ of permutations of $X$; \\
	We remark that any such map defines an action of $MF_P$ on $X$ via $s_1 \ldots s_n \cdot x := \phi(s_1) \circ \ldots \circ \phi(s_n)(x)$, where $s_i \in \sS \cup \sS^{-1}$, and $\phi(s^{-1}) := \phi(s)^{-1}$. We require that
		\begin{enumerate}
		\item \label{phi II} this action of $MF_P$ on $X$ is transitive, 
			 and 
	\item \label{phi I} for all $s \in \sI$ the permutation $\phi(s)$ is fixed point free of order 2;
		\end{enumerate} 
		\item \label{SP iv}
		a \defin{relator set} $\sR_x \subset Stab(MF_P,x) = $ for each $x \in X$, where $Stab(MF_P,x)$ denotes the stabiliser of $x$ with respect to the aforementioned action of $MF_P$. (This is a natural condition, as we want to return to our starting vertex when following a walk labelled by a relator, and in particular we want to return to the same vertex class.)\\
		 The set $\{ \sR_x : x\in X\}$ of these relator sets is denoted by $\sR$.
	\end{enumerate}

\end{defn}

We now use such a presentation $P = \langle X \vert \sU \vert \sI \vert \phi \vert \sR \rangle$ to define the \pcg\ $\Sp(P)$, in analogy with \Dr{spectopdef}.
We start by defining the \defin{presentation graph} $C(P)$. This has vertex set $X$, and directed edge set $\{ (x, \phi(s)(x) \vert$ for all $x \in X$ and $s \in \sU \cup (\sU)^{-1} \cup \sI\}$ where $\phi(s^{-1}) = \phi(s)^{-1}$. We colour it by  $c: \overrightarrow{E}(C(P)) \rightarrow \sS \cup \sS^{-1}$ defined by $c(x, \phi(s)x) := s$, and note that this is a Cayley-like colouring as in \Dr{def Cl}.

	The \defin{\pp\ complex} $\sC(P)$ is the 2-complex obtained from $C(P)$ as follows. For each $x\in X$ and each $r \in \sR_x$,  we introducing a 2-cell and glue its boundary along the walk of  $C(P)$ starting at $x$ and dictated by $r$ (as in  \Dr{def dict}). It is straightforward to check that this is a closed walk using \eqref{SP iv}.

	Note that $\sC(P)$ is connected by condition \eqref{phi II}. Finally, 
	
\begin{defn} \label{def SpP} 
	We define the \defin{\pcg} $\Sp(P)=\Sp\langle X \vert \sU \vert \sI \vert \phi \vert \sR \rangle$ to be the 1-skeleton of the universal cover of $\sC(P)$. 
\end{defn}

	Letting $\epsilon: \Sp(P) \rightarrow C(P)$ be the covering map, we can lift $c$ to the edge-colouring $\tilde{c} =c \circ \epsilon$ of $\Sp(P)$.

	Note that if $X$ is a singleton, then we recover the usual group presentations and \Cg s by the above definitions. Our \bpp s $\langle \sS_1, \sU', \sI' \vert \sR_0, \sR_1 \rangle$ of \Sr{sec special} are tantamount to \pp s as in \Dr{def GSP} with $X = \{0,1\}$, where $\phi(s_1) = (0)(1)$ for $s_1 \in \sS_1$ and $\phi(s_2) = (0,1)$  for $s_2 \in \sS_2:= \sU' \cup \sI'$, with $\sU = \sS_1 \cup \sU'$ and $\sI = \sI'$.

\mymargin{Needed?: Then $\pi_1(C(P), x) = Stab(MF_P,x)$ using the colouring $c$, therefore $\sR_x \subset Stab(MF_P,x) = \pi_1(C(P),x)$.}

\bigskip
As in \Sr{sec special}, we can alternatively define  $\Sp(P)$ as a graph quotient, following the lines of Definition \ref{specgrpdef}, as follows:
\begin{enumerate}
	\item Let $\sS := \sU \cup \sI$ and define the group $MF_P$ by the presentation $ \langle \sS \vert \{s^2 : s \in \sI \} \rangle$; this is a free product of infinite cyclic groups, one for each $s\in \sU$, and cyclic groups of order 2, one  for each $s\in \sI$. Define the tree $T_P$  by
	\begin{align*}
		V(T_P) & := MF_P, \mbox{ and }\\
		\overrightarrow{E}(T_P) & := \{(w, ws) \vert w \in MF_P, s \in \sS \cup \sS^{-1}\}.
	\end{align*}
	This is a  $(2\vert \sU \vert + \vert \sI \vert)$-regular tree, and it comes with a colouring $c: \overrightarrow{E}(T_P) \rightarrow \sS \cup \sS^{-1}$ by $c(w, ws) = s$.
	\item 
	We can extend the map $\phi$ of (3) from $\sS$ to an action of $MF_P$ by composition: we let $x \cdot s_1 \ldots s_n := \phi(s_n) \circ \ldots \circ \phi(s_1)(x)$ for all  $x\in X$ and $s_i\in \sS$.  Let $W_{x,y} = \{w \in MF_P \vert \phi(w)(x) = y\}$ for $x,y \in X$. Fixing any `base' vertex class $b \in X$ leads to a partition of $V(T_P) = MF_P$, namely $\tilde{V}_{x} = W_{b,x}$. Note that two vertices in $u,v \in \tilde{V}_x \subset MF_P$ differ by a word $u^{-1}v \in W_{x,x} = Stab(MF_P,x)$.
	\item Let $R_x = \langle wrw^{-1} \vert r \in \sR_y, w \in W_{x,y}, y \in X \rangle \subset W_{x,x}$. Then we say that two vertices in $u,v \in \tilde{V}_x$ are equivalent, and write $u \sim v$, if $u^{-1}v \in R_{x}$. Similarly, for edges $e, f \in \overrightarrow{E}(T_P)$ we write $e \sim f$ if $c(e) = c(f)$ and $\tau(e) \sim \tau(f)$ and $\tau(e^{-1}) \sim \tau(f^{-1})$.
	\item We define $\Sp(P)$ to be the corresponding quotient  $T_P/\sim$.
\end{enumerate}

As in \Cr{part}, it is not hard to see that $T_P$ is the universal cover of $\Sp(P)$. Define $V_x, x \in X$ as the image of $\tilde{V_x}$ under the quotient of $\sim$. We have $W_{x,x} = \pi_1(C(P),x)$ and $\pi_1(\sC(P),x) = R_x \backslash W_{x,x} =: G_x$, analogously to the \bpp\ case. We call $G_x, x \in X$ the \defin{vertex groups}. 

We remark that the vertex set of $\Sp(P)$ can be given the structure of a groupoid $\sG_{\Sp(P)}$. Indeed, we can think of $\bigcup_{x,y \in X} W_{x,y}$ as the ground set, and define the groupoid operation $W_{x,y} \times W_{y,z} \rightarrow W_{x,z}$ by concatenation. Another way to think of this groupoid is $\sG_{\Sp(P)} \cong \pi_1(\sC(P), X)$, the universal groupoid of the presentation complex $\sC(P)$, with paths starting and ending in $V(C)$. 

\medskip
The main result of this section is that every vertex transitive graph $\Gha$ is isomorphic to $\Sp(P)$ for some \pp\ $P$. For the proof of this we will need to decompose the edges of $\Gha$ into cycles. The next section discusses such decompositions.

\subsection{Multicycle colourings} \label{sec mcc}

Leighton \cite{Le83} asked whether vertex transitive graphs have similar colouring structures to Cayley graphs of groups. For a Cayley graph $\Gha = \Cay(\Gpa, \sS)$, the generators canonically induce a colouring $c: E(\Gha) \rightarrow \sS$ as above, so that $c^{-1}(s)$ is a disjoint union of cycles of the same length for every $s\in \sS$. Leighton calls this a multicycle:

\begin{defn} \label{def mc}
	A \defin{multicycle} is a graph which is either the disjoint union of cycles of the same length or a perfect matching. A \defin{multicycle colouring} of a  graph $\Gha$ is a colouring $c: E(\Gha) \rightarrow \Omega$ such that the graph with vertex set $V(\Gha)$ and edge set $c^{-1}(x)$ is a multicycle for each $x \in \Omega$.
\end{defn}
Thus every Cayley graph has a multicycle colouring. Leighton  \cite{Le83} conjectured that all vertex transitive graphs have a multicycle colouring \cite{Le83}, but this was shown to be false by Maru\v{s}i\v{c} \cite{Ma81}, a counter-example being the line graph of the Petersen graph:

\begin{examp} \label{Line-Graph-Petersen}
	Given a graph $\Ghb$ we construct the line graph $\Gha := L(\Ghb)$ as follows. We set $V(\Gha) := E(\Ghb)$ and $\overrightarrow{E}(\Gha) = \{(e,e') \vert \tau(e^{\pm 1}) = \tau(e'^{\pm 1})\}$. To see there is no multicycle colouring of $L(P(5,2))$, note that it has $\vert V(L(P(5,2))) \vert = \vert E(P(5,2)) \vert = 15$ vertices, so any mutlicycle will have to consist of triangles, pentagons, or 15-cycles. Any 15-cycle in $L(P(5,2))$ would yield a Hamiltonian cycle in $P(5,2)$, which we know does not exist. Moreover, the only triangles in $L(P(5,2))$ are formed by edges incident with a single vertex of $P(5,2)$. As $P(5,2)$ is not bipartite, there is no way to partition the triangles into disjoint sets that pass through all vertices. So we can only use sets of five cycles, which correspond to sets of edge disjoint pentagons in $P(5,2)$. As $P(5,2)$ is cubic, there is no set of pentagons that visits every edge exactly once. 
	
	Still, it is possible to express  $L(P(5,2))$ as a \pcg:\\ $\mbox{\Spl}\langle\textcolor{red}{a} \mapsto (12)(3), \textcolor{blue}{b} \mapsto (1)(23) \vert \{b^5, a^{10}, a^2b\}, \{a^{-2}b^4\}, \{a^5, b^{10}, b^2a\} \rangle$
	\showFigTikz{
		\begin{tikzpicture}[scale=0.3]
		\node[circle,fill=black,scale = 0.3] (v1) at (-0.5,3.5) {};
		\node[circle,fill=black,scale = 0.3] (v4) at (-2,2.5) {};
		\node[circle,fill=black,scale = 0.3] (v3) at (1,2.5) {};
		\node[circle,fill=black,scale = 0.3] (v2) at (-1.5,1) {};
		\node[circle,fill=black,scale = 0.3] (v5) at (0.5,1) {};
		\node[circle,fill=black,scale = 0.3] (v9) at (-0.5,6) {};
		\node[circle,fill=black,scale = 0.3] (v6) at (-3,-1) {};
		\node[circle,fill=black,scale = 0.3] (v7) at (2,-1) {};
		\node[circle,fill=black,scale = 0.3] (v10) at (-4,4) {};
		\node[circle,fill=black,scale = 0.3] (v8) at (3,4) {};
		\node[circle,fill=black,scale = 0.3] (v14) at (-4,7) {};
		\node[circle,fill=black,scale = 0.3] (v15) at (3,7) {};
		\node[circle,fill=black,scale = 0.3] (v11) at (-6,0.5) {};
		\node[circle,fill=black,scale = 0.3] (v13) at (5,0.5) {};
		\node[circle,fill=black,scale = 0.3] (v12) at (-0.5,-3.5) {};
		\draw[->,blue]  (v1) edge (v2);
		\draw[->,blue]  (v2) edge (v3);
		\draw[->,blue]  (v3) edge (v4);
		\draw[->,blue]  (v4) edge (v5);
		\draw[->,blue]  (v5) edge (v1);
		\draw[->,red]  (v6) edge (v4);
		\draw[<-,red]  (v6) edge (v5);
		\draw[->,red]  (v7) edge (v2);
		\draw[<-,red]  (v7) edge (v3);
		\draw[->,red]  (v8) edge (v5);
		\draw[<-,red]  (v8) edge (v1);
		\draw[->,red]  (v9) edge (v3);
		\draw[->,red]  (v4) edge (v9);
		\draw[->,red]  (v10) edge (v1);
		\draw[<-,red]  (v10) edge (v2);
		\draw[->,blue]  (v6) edge (v11);
		\draw[->,red]  (v11) edge (v12);
		\draw[->,blue]  (v12) edge (v6);
		\draw[<-,blue]  (v12) edge (v7);
		\draw[<-,blue] (v7) edge (v13);
		\draw[<-,red]  (v13) edge (v12);
		\draw[->,blue]  (v11) edge (v10);
		\draw[<-,red] (v11) edge (v14);
		\draw[<-,blue]  (v14) edge (v10);
		\draw[->,blue]  (v14) edge (v9);
		\draw[->,blue]  (v9) edge (v15);
		\draw[->,red]  (v15) edge (v14);
		\draw[->,blue]  (v15) edge (v8);
		\draw[->,blue]  (v8) edge (v13);
		\draw[->,red]  (v13) edge (v15);
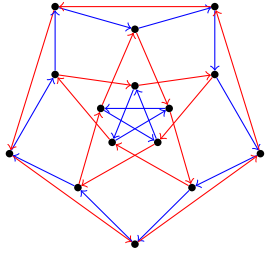
		\end{tikzpicture}}{petersen-line-graph}{The line graph $L(P(5,2))$ of the Petersen graph.}
\end{examp}

Our aim now is to weaken the notion of a multicycle colouring enough that every vertex transitive graph will admit one, so that the weakened notion will allow us to find \pp s. This is the essence of \Tr{colour-induced-split-presentation} below.

\begin{defn} \label{def wmc}
	A graph $\Gha$ is a \defin{weak multicycle}, if it is a vertex-disjoint union of cycles and edges. A \defin{weak multicycle colouring} of a  graph $\Gha$ is a colouring $c: E(\Gha) \rightarrow \Omega$ such that the graph with vertex set $V(\Gha)$ and edge set $c^{-1}(x)$ is a weak multicycle for each $x \in \Omega$.
\end{defn}

We say that a weak multicycle colouring $c$ is \defin{\pf}, if $c^{-1}(x)$ is regular for all $x \in \Omega$. In other words, $c^{-1}(x)$ is either a disjoint union of cycles or a perfect matching for all $x$.


As we will see in the following section, every vertex transitive graph has a \pf\ weak multicycle colouring. The condition of vertex transitivity here cannot be relaxed to just regularity. Indeed, let $\Gha$ be the 3-regular graph in Figure~\ref{figReg}. Since its vertex degrees are odd, one of the colours in any weak multicycle colouring must induce a perfect matching. But $\Gha$ does not have a perfect matching $M$, because removing $v$ and the vertex matched to $v$ by $M$ results in at least one component with an odd number of vertices. 

	\showFigTikz{
\begin{tikzpicture}[scale = 0.4]
		\node (G-00) at (0,0) [circle,fill=black,scale = 0.3, label=right:$v$] {};
		\node (G-11) at (1,-1) [circle,fill=black,scale = 0.3] {};
		\node (G-12) at (3,-1) [circle,fill=black,scale = 0.3] {};
		\node (G-13) at (1,-3) [circle,fill=black,scale = 0.3] {};
		\node (G-14) at (4,-2) [circle,fill=black,scale = 0.3] {};
		\node (G-15) at (3,-2) [circle,fill=black,scale = 0.3] {};
		\node (G-16) at (2,-3) [circle,fill=black,scale = 0.3] {};
		\node (G-17) at (2,-4) [circle,fill=black,scale = 0.3] {};
		\node (G-21) at (-1,-1) [circle,fill=black,scale = 0.3] {};
		\node (G-22) at (-3,-1) [circle,fill=black,scale = 0.3] {};
		\node (G-23) at (-1,-3) [circle,fill=black,scale = 0.3] {};
		\node (G-24) at (-4,-2) [circle,fill=black,scale = 0.3] {};
		\node (G-25) at (-3,-2) [circle,fill=black,scale = 0.3] {};
		\node (G-26) at (-2,-3) [circle,fill=black,scale = 0.3] {};
		\node (G-27) at (-2,-4) [circle,fill=black,scale = 0.3] {};
		\node (G-31) at (0,1.2) [circle,fill=black,scale = 0.3] {};
		\node (G-32) at (1.4,2.6) [circle,fill=black,scale = 0.3] {};
		\node (G-33) at (-1.4,2.6) [circle,fill=black,scale = 0.3] {};
		\node (G-34) at (1.4,4) [circle,fill=black,scale = 0.3] {};
		\node (G-35) at (0.7,3.3) [circle,fill=black,scale = 0.3] {};
		\node (G-36) at (-0.7,3.3) [circle,fill=black,scale = 0.3] {};
		\node (G-37) at (-1.4,4) [circle,fill=black,scale = 0.3] {};
		\draw[-] (G-11) edge (G-00);
		\draw[-] (G-21) edge (G-00);
		\draw[-] (G-31) edge (G-00);
		\draw[-] (G-11) edge (G-12); 
		\draw[-] (G-11) edge (G-13);
		\draw[-] (G-12) edge (G-14);
		\draw[-] (G-12) edge (G-15);
		\draw[-] (G-13) edge (G-16);
		\draw[-] (G-13) edge (G-17);
		\draw[-] (G-14) edge (G-15);
		\draw[-] (G-14) edge (G-17);
		\draw[-] (G-15) edge (G-16);
		\draw[-] (G-16) edge (G-17);
		\draw[-] (G-21) edge (G-22); 
		\draw[-] (G-21) edge (G-23);
		\draw[-] (G-22) edge (G-24);
		\draw[-] (G-22) edge (G-25);
		\draw[-] (G-23) edge (G-26);
		\draw[-] (G-23) edge (G-27);
		\draw[-] (G-24) edge (G-25);
		\draw[-] (G-24) edge (G-27);
		\draw[-] (G-25) edge (G-26);
		\draw[-] (G-26) edge (G-27);
		\draw[-] (G-31) edge (G-32); 
		\draw[-] (G-31) edge (G-33);
		\draw[-] (G-32) edge (G-34);
		\draw[-] (G-32) edge (G-35);
		\draw[-] (G-33) edge (G-36);
		\draw[-] (G-33) edge (G-37);
		\draw[-] (G-34) edge (G-35);
		\draw[-] (G-34) edge (G-37);
		\draw[-] (G-35) edge (G-36);
		\draw[-] (G-36) edge (G-37);
\end{tikzpicture}}{figReg}{A regular graph with no \pf\ weak multicycle colouring.}

\subsection{Multicycle colourings and \pp s}

We say a \pp\ $P = \langle X \vert \sU \vert \sI \vert \phi \vert \sR \rangle$ is \defin{uniform}, if for every $s \in \sS$, all orbits of $\phi(s)$ have  the same size. In other words, if $c$ is a multicycle colouring on $C(P)$. In light of Leighton's aforementioned conjecture, one can ask the following:

\begin{que}
	Let $\Gha$ be a vertex transitive graph. Does $\Gha$ have a multi-cycle colouring if and only if it is the \pcg\ of a uniform \pp?
\end{que}

The forward direction is true: if $\Gha$ has a multicycle colouring then it has a uniform \pp\ given in the proof of Theorem \ref{colour-induced-split-presentation}. 
But the backward direction is false, as shown by the following example.  	Consider the \bpp\ $P = \langle \{a\}, \{b\}, \emptyset \vert \{a\}, \{a^2\} \rangle$. This is trivially uniform, like every \bpp. However, $\Sp(P)$, shown in Figure~\ref{multi-cycle-counter-example},  does not have a multicycle colouring.
	\showFigTikz{
		\begin{tikzpicture}[scale=2]
		\node[circle,fill=black,scale = 0.3] (v1) at (0,0) {};
		\node[circle,fill=black,scale = 0.3] (v2) at (1,0) {};
		\node[circle,fill=black,scale = 0.3] (v3) at (2,0) {};
		\node[circle,fill=black,scale = 0.3] (v4) at (3,0) {};
		\draw[->,blue]  (v1) edge [bend left] (v2);
		\draw[->,blue]  (v2) edge [bend left] (v1);
		\draw[->,blue]  (v3) edge [bend left] (v4);
		\draw[->,blue]  (v4) edge [bend left] (v3);
		\draw[->,red]  (v2) edge [bend left] (v3);
		\draw[->,red]  (v3) edge [bend left] (v2);
		\draw[->,red] (v1)  edge [loop left] (v1);
		\draw[<-,red] (v4) edge [loop right] (v4);
		\end{tikzpicture}}{multi-cycle-counter-example}{$\Sp\langle \{a\}, \{b\}, \emptyset \vert \{a\}, \{a^2\} \rangle$}

The following result will be used later to show that every vertex transitive graph admits a \pp.
 
\begin{thm} \label{colour-induced-split-presentation}
A connected graph has a \pf\ weak multicycle colouring if and only if it admits a \pp.
\end{thm}
\begin{proof}
	Recall that a graph is defined using a directed edge set $\overrightarrow{E}(\Gha)$, but we can also consider the undirected edge set $E(\Gha) = \overrightarrow{E}(\Gha)/\,^{-1}$, so that an undirected edge is a pair $\{e,d\}$ such that $e^{-1} = d$ and $d^{-1} = e$. In the following proof we have to transition between colourings of the directed edges and colourings of the undirected edges. Apart from this, the proof boils down to a straightforward checking of the conditions of the corresponding definitions.

\medskip
	For the forward direction, suppose $\Gha$ is connected and it has a \pf\ weak multicycle colouring $c: E(\Gha) \rightarrow \Omega$. To define the desired \pp\ $P$, we start with
	\begin{itemize}
		\item $X = V(\Gha)$,
		\item $\sU = \{\omega \in \Omega \vert c^{-1}(\omega)$ is of degree 2$\}$, and
		\item $\sI = \{\omega \in \Omega \vert c^{-1}(\omega)$ is of degree 1$\}$.
	\end{itemize}
	Since $c$ is \pf, we have $\sU \cup \sI = \Omega$.
	We want to refine $c$ into a colouring  $c'$ of the directed edges  of $\Gha$. To do this, for each $\omega \in \sU$ we choose an orientation  $O_{\omega}\subset \overrightarrow{E}(\Gha) $ of $c^{-1}(\omega)\} \subset {E}(\Gha)$ (recall this means that $(O_{\omega} \cup O_{\omega}^{-1})/\,^{-1} = c^{-1}(\omega)$ and $O_{\omega} \cap O_{\omega}^{-1} = \emptyset$). Since $c^{-1}(\omega)$ is a multicycle, we can choose  $O_{\omega}$ so that each of its cycles is oriented, that is, for each vertex $v\in V(\Gha)$ there is exactly one $e\in O_{\omega}$ with $\tau(e)=x$. Thus $O_{\omega}$ defines a permutation $\phi(\omega)$ of $X = V(\Gha)$, by letting $\phi(\omega)(x)$ be the unique $y\in X$ such that $(x,y)\in O_{\omega}$. Moreover, for each $\omega \in \sI$, let $O_{\omega} = \{e \in \overrightarrow{E}(\Gha) \vert [e] \in c^{-1}(\omega)\}$, and let $\phi(\omega)$ be the involution of $V(\Gha)$ exchanging the end-vertices of each edge in $c^{-1}(\omega)$. Thus $\phi$ satisfies \eqref{phi I} of \Dr{def GSP} by construction (we will check \eqref{phi II} below).
	
	We now define $c'$ by
	\[
	c'(e) = \begin{cases} c([e]) & \mbox{ if } e \in O_{c([e])}\\ c([e])^{-1} & \mbox{ otherwise. } \end{cases}
	\] 
	This maps $E(\Gha)$ to $ \sU \cup \sU^{-1} \cup \sI$, because for $e \in \overrightarrow{E}(\Gha)$ such that $c([e]) \in \sI$ we have $e,e^{-1} \in O_{c([e^{-1}])}$ by definition. Easily, $c'$ is a Cayley-like colouring. This allows us to define $\sW_v$ on $\Gha$ as described after \Dr{def Cl}. Note that as $\Gha$ is connected, for any two $x,y \in V(\Gha)$ there is a path $p$ connecting $x$ and $y$. Then the path $p$ corresponds to a word $\sW_x(p) \in MF_P$ such that $\phi(\sW_x(p))(x) = y$. Therefore the action of $MF_P$  on $X = V(\Gha)$ is transitive as required by \eqref{phi II} of \Dr{def GSP}.
	
	To complete  the definition of our \pp\ $P$, we choose the relators
	\begin{itemize}
		\item $\sR_v = \sW_v(\pi_1(\Gha,v)) \subset MF_P$.
	\end{itemize}
	We claim that $\Gha$ coincides with the presentation graph $C(P)$. To begin with, they have the same vertex set $V(C) = X = V(\Gha)$. Moreover,
	\begin{align*}
	\overrightarrow{E}(C(P)) & = \{(x,\phi(\omega)(x)) \vert x \in V(\Gha), \ \omega \in \sU \cup \sU^{-1} \cup \sI\}\\
	& = \cup_{\omega \in \Omega} \{(x,y) \vert (x,y) \in O_\omega \mbox{ or } (y,x) \in O_\omega\}\\ 
	& = \cup_{\omega \in \Omega} (c')^{-1}(\omega) = \overrightarrow{E}(\Gha)
	\end{align*}
	and so our claim is proved.
	
	As we defined $\sC(P)$ by glueing in a 2-cell along each closed walk dictated by an element of $\sR_v, v\in V(C(P))$, where we have chosen $\sR_v = \sW_v(\pi_1(\Gha,v))$, we have forced $\pi_1(\sC(P),v)$ to be trivial. Therefore,  $\sC(P)$ coincides with its own universal cover $\widehat{\sC(P)}$. Thus $\Sp(P)$, defined as the 1-skeleton of $\widehat{\sC(P)}$, is $C(P) = \Gha$. Therefore $P$ is a \pp\ for $\Gha$. 
	
\medskip	
	For the converse direction, let $\Gha = \Sp(P)$ for some \pp\ $P$. Let $\epsilon : \Gha \rightarrow C(P)$ be the  covering map, and $c_C': \overrightarrow{E}(C) \rightarrow \sU \cup \sU^{-1} \cup \sI$ the colouring induced by the generators of $P$, as in the definition of $\Sp(P)$. We collapse $c_C'$ into a colouring $c_C$ of the undirected edges of $C$ defined by
	\[
	c_C([e]) = \begin{cases} u \in \sU & \mbox{ if } c({e}) \in \{u, u^{-1}\} \\
			 i \in \sI & \mbox{ if } c({e}) = i. 
			 \end{cases}
	\]
	We can collapse  $c_{\Gha}': \overrightarrow{E}(\Gha) \rightarrow \sU \cup \sU^{-1} \cup \sI$ similarly to obtain an undirected colouring $c_{\Gha}: E(\Gha) \rightarrow \sU \cup \sI$. Note that $c_{C}$ is a \pf\ weak multicycle colouring, with $c^{-1}(i)$ being of degree  1 for $i \in \sI$ and $c^{-1}(u)$ being of degree 2 for $u \in \sU$, by the  definitions. As $c_{C}' \circ \epsilon = c_{\Gha}'$, it is easy to verify that $c_C \circ \epsilon = c_{\Gha}$. This implies that $c_{\Gha}^{-1}(x)$ has the same degree as $c^{-1}_{C}(x)$, and that every vertex has at least one incident edge coloured $s$ for each $s \in \sI \cup \sU$. This means that $c_{\Gha}$ is a \pf\ weak multicycle colouring of $\Gha$ as claimed. 
\end{proof}

\subsection{Weak multicycle colourings of vertex transitive graphs} \label{sec mc trans}

The aim of this section is to show that every  vertex transitive graph $\Gha$ has a \pf\ weak multicycle colouring, hence it admits a \pp\ by \Tr{colour-induced-split-presentation}.

For this, we will use the following result of Godsil and Royle \cite[Theorem 3.5.1]{GR01}:

\begin{thm}[{Godsil \& Royle \cite[Theorem 3.5.1]{GR01}}] \label{vertex-transitive-matching}
	 Let $\Gha$ be a connected finite vertex transitive graph. Then $\Gha$ has a matching that misses at most one vertex.
\end{thm}

In the Appendix we generalise this to infinite vertex transitive graphs as follows

\paragraph{\textbf{Theorem \ref{inf-vertex-matching}}}
	Let $\Gha$ be a countably infinite, connected, vertex transitive graph. Then $\Gha$ has a perfect matching.\\

In passing, let us mention the following still open conjecture. If true, it would imply that all finite vertex transitive cubic graphs have a uniform \pp. 

\begin{conj}[{Lovasz \cite[Problem 11]{LL70}}]
	 Let $\Gha$ be a finite cubic vertex transitive graph. Then there exists a perfect matching $M$ in $\Gha$ such that $\Gha$ $\backslash$ $M$ consists of either one cycle, (and $\Gha$ is Hamiltonian), or of two disjoint cycles of the same length.
\end{conj}

The following old theorem of Petersen is a rather straightforward application of Hall's Marriage theorem \cite{PH35}. Although it is well-known, we include a proof for convenience.

\begin{thm}[{J.~Petersen \cite{Petersen}}] \label{2-regular-sub-graph}
	 Every regular graph of positive (finite and) even degree has a spanning 2-regular subgraph.
\end{thm}
\begin{proof}
	Let $\Gha$ be a 2$k$-regular graph. If $\Gha$ is finite then it contains an Euler tour $C$ (i.e.\ a closed walk that uses each edge exactly once) by Euler's theorem \cite{DiestelBook05}. Pick an orientation of $O_C \subset \overrightarrow{E}(\Gha)$ of $C$. If $\Gha$ is infinite then just choose an orientation with equal in and out degree, which can be constructed greedily. Then construct an auxiliary graph $\Ghb$ with
	\begin{align*}
	V(\Ghb) = & \{v^+, v^i \vert v \in V(\Gha)\}, \ \mbox{and}\\
	E(\Ghb) = & \{(v^+, u^i) \vert (v,u) \in O_C\}.
	\end{align*}
	By definition, $\Ghb$ is $k$-regular and bipartite, with bipartition $V^+ = \{v^+ \vert v \in V(\Gha)\}$ and $V^- = \{v^- \vert v \in V(\Gha)\}$. For any finite $A \subset V^+$, as $\Ghb$ is $k$-regular, the neighbourhood $N(A) = \{ u^- \vert (v^+, u^-) \in E(\Ghb) \ \mbox{with} \ v^+ \in A\}$ of $A$ has size at least $k \times \vert A \vert / k = \vert A \vert$. So by Hall's Marriage theorem \cite{PH35}, $\Ghb$ contains a perfect matching $M \subset E(\Ghb)$. Then the spanning subgraph $S \subset \Gha$ given by $\{(v,u) \vert (v^+,u^-) \in M\} \subset E(\Gha)$ is 2-regular by construction.
\end{proof}

Combining this with Theorem \ref{vertex-transitive-matching} and Theorem \ref{inf-vertex-matching}, we now obtain 

\begin{lem} \label{sfwmc}
	Every countable, vertex transitive, graph $\Gha$ has a \pf\ weak multicycle colouring.
\end{lem}
\begin{proof}
	We first consider the case where $\Gha$ is (finite or) locally finite. As $\Gha$ is vertex transitive it is $n$-regular for some $n\in \mathbb{N}$. If $n$ is even, then we can apply  Theorem \ref{2-regular-sub-graph} recursively to decompose $E(\Gha)$ into $2$-regular spanning subgraphs, and attributing a distinct colour to the edges of each of those subgraphs yields a \pf\ weak multicycle colouring. 
	
	If $n$ is odd, then we first find a perfect matching $M$, colour its edges with the same colour, and treat $\Gha \backslash M$ as above to obtain a \pf\ weak multicycle colouring. To obtain $M$, note that if $\Gha$ is finite, then $\vert V(G) \vert$ is even since $\vert E(\Gha) \vert = n \vert V(G) \vert /2$. Therefore $\Gha$ has a perfect matching by Theorem \ref{vertex-transitive-matching} as no matching can miss exactly 1 vertex in this case. If $\Gha$ is infinite, then Theorem \ref{inf-vertex-matching} provides a perfect matching. 
	
	If $\Gha$ is not locally finite, then each vertex has countably infinite degree. We will decompose $E(\Gha)$ into an edge-disjoint union of multicycles $\{M_i\}_{i\in\mathbb{N}}$, where each $M_i$ is a spanning subgraph consisting of pairwise vertex-disjoint double-rays. For this, let $\{e_i\}_{i\in\mathbb{N}}$ be an enumeration of the edges of $\Gha$, and let $\{v_i\}_{i\in\mathbb{N}}$ be a sequence of vertices of $\Gha$ in which each $v\in V(\Gha)$ appears infinitely often. We greedily construct an  $M_0$ as above containing $e_0$ as follows. 
We start with $M_0^0=e_0$, and for $i=1,2,\ldots$, we extend the (possibly trivial) path in $M_0^{i-1}$ containing $v_i$ into a longer path by adding an edge of $\Gha- M_0^{i-1}$ at each of its end-vertices. As $M_0^{i-1}$ is finite, and every vertex has infinite degree, this is always possible. Finally, we let $M_0:= \bigcup_{i\in \mathbb{N}} M_0^i$. Since each $v\in V(\Gha)$ appears infinitely often as $v_i$, we deduce that $M_0$ is a spanning union of vertex-disjoint double-rays as desired. 

Having constructed $M_0$, we inductively construct the $M_i, i\geq 1$ so that $M_i$ contains $e_i$ unless $e_i$ is already in $\bigcup_{j<i} E(M_j)$, by noticing that $\Gha - \bigcup_{j<i} E(M_j)$ is a regular graph with countably infinite degree itself, and repeating the above proceedure. Then $\{M_i\}_{i\in\mathbb{N}}$ is the desired \pf\ weak multicycle colouring of $\Gha$.
\end{proof}

This combined with Theorem \ref{colour-induced-split-presentation} yields one of our main results:

\paragraph{\textbf{Theorem \ref{finite-vertex-presentation}}}
	Every 
	countable, vertex transitive, graph has a \pp.\\

We conclude this section with the following question.

\begin{que}
	For a vertex transitive graph $\Gha$ with \pp\ $P$ does $\Aut_{c-loc}(\Spz)$ act vertex transitively on $\Gha = \Sp(P)$ where $c$ is the colouring coming from $P$?
\end{que}

\subsection{Generalised results}

Here we extend some of our earlier results from 2-partite to general \pp s. Where the same arguments apply directly the proofs will be omitted. First we generalise Lemma \ref{verttransonV}:

\begin{prop} \label{genloctrans}
	For a \pp\ $P = \langle X \vert \sU \vert \sI \vert \phi \vert \sR \rangle$ there is a natural inclusion of the vertex group $G_x \leq \Aut_c(\Sp(P))$ for each $x \in X$. Moreover $G_x$ acts regularly on $V_x$, and so $\Sp(P)$ is $\vert X \vert$-Cayley. 
\end{prop}

The vertex groups are still isomorphic due to the fact that $\pi_1$ does not depend on the choice of a base point:

\begin{prop}
	For every \pp\ $P = \langle X \vert \sU \vert \sI \vert \phi \vert \sR \rangle$, and every $x, y \in X$, the vertex groups $G_{x}, G_{y}$ are isomorphic.
\end{prop}
\begin{proof}
As above, let  $C(P) =: C$ be the presentation graph of $P$. Let $x,y \in X = V(C)$. Recall that $G_x := W_{x,x}/ R_x$ is the right quotient of $W_{x,x}$ by $R_x$, where $W_{x,z}$ is the set of paths in $C$ from $x$ to $z$ up to homotopy (in particular, $W_{x,x} = \pi_1(C,x)$), and\\ 
$$R_x := \langle \{ w\sW^{-1}_z(r)w^{-1} \vert r \in \sR_z, w \in W_{x,z} , z \in X\} \rangle$$ with $\sW^{-1}_z$ the map from words in $MF_P$ to paths in $C$ defined in section 2. Let $p \in W_{x,y}$ be a path from $x$ to $y$ in $C$. As $\pi_1$ is base point preserving, we have
\begin{align} \label{pi-one-preservation}
W_{x,x} = p W_{y,y} p^{-1}.
\end{align}
Moreover, note that
\begin{align*}
pR_yp^{-1} = & \langle \{ (pw)\sW_z^{-1}(r)(pw)^{-1} \vert r \in \sR_z, w \in W_{y,z} , z \in X \} \rangle\\
=& \langle \{ w'\sW_z^{-1}(r)(w')^{-1} \vert r \in \sR_z, w' \in W_{x,z} , z \in X \} \rangle\\
=& R_x.
\end{align*}
This defines a homomorphism $\phi: G_y \rightarrow G_x$ by $\phi: w R_y \mapsto p w p^{-1} R_x$ for every $w\in W_{y,y}$. It is surjective by (\ref{pi-one-preservation}) and injective as $pR_yp^{-1} = R_x$. Thus it is an isomorphism proving our claim.
\end{proof}

We generalise Proposition~\ref{prestrans} to obtain a sufficient condition for vertex transitivity.

\begin{prop} \label{simplicial-automorphisms-lift}
	Let $P = \langle X \vert \sU \vert \sI \vert \phi \vert \sR \rangle$ be a \pp. If the  presentation complex $\sC(P)$ is vertex transitive, then so is $\Sp(P)$.
\end{prop}

\subsection{Quasi-isometry to vertex groups}

It is a straightforward consequence of the \v{S}varc--Milnor lemma \cite{Mi68} that if $P = \langle X \vert \sU \vert \sI \vert \phi \vert \sR \rangle$ is a \pp\ with finite $X$, then $\Gha := \Sp(P)$ is quasi-isometric to (any \Cg\ of) $G_x$. In this section we provide the details for the non-expert reader. This will be used in \Sr{sec Conc} to argue that there are \pcg s that cannot be represented by a \pp\ with finite $X$.

\medskip
A \defin{quasi-isometry} between metric spaces $(X,d)$ and $(Y,d')$ is a (not necessarily continuous) function $f: X \to Y$ satisfying the following two statements for some constants $A,B\in \mathbb{R}_+$:
$$ \frac1{A} d(x,z) - B \leq d'(f(x),f(z) \leq A d(x,z) + B$$
for every $x,z\in X$, and
for every $y\in  Y$ there is $x\in X$ such that $d'(f(x),y)\leq B$.

If such an $f$ exists, we say that $(X,d)$ and $(Y,d')$ are \defin{quasi-isometric} to each other. (Easily, this is an equivalence relation.)
It is well-known, and easy to check, that any two finitely generated \Cg s of the same group are quasi-isometric to each other. We say that a metric space $(X,d)$ is \defin{quasi-isometric} to a group $G$, if $(X,d)$ is {quasi-isometric} to some, hence to every, finitely generated \Cg\ of $G$.

\begin{prop} \label{quasi}
Let $P = \langle X \vert \sU \vert \sI \vert \phi \vert \sR \rangle$ be a \pp\ with finite $X$. Then $\Gha := \Sp(P)$ is quasi-isometric to $G_x$ for every $x \in X$.
\end{prop}
\begin{proof}
Consider the inclusion map $i: C \rightarrow \sC$ from the presentation graph $C := C(P)$ to the presentation complex $\sC := \sC(P)$ of $P$. It is well-known  \cite[Proposition 1.26]{HA02}  that the inclusion of the one skeleton into a 2-complex induces a surjection on the level of fundamental groups, and the kernel is exactly the normal closure of the words bounding the 2-cells. Thus $i_{\ast} : \pi_1(C,x) \rightarrow \pi_1(\sC,x)$ is a surjection with kernel $R_x$, so that $\pi_1(\sC,x) = \pi_1(C,x)/R_x = W_{x,x}/R_x = G_x$. 

Let $\widehat{\Gha}$ be the universal cover of $\sC$, with covering map $\widehat{\eta} : \widehat{\Gha} \rightarrow \sC$. As $\pi_1(\sC,x) = G_x$ we have an action of $G_x$ on $\widehat{\Gha}$ ---and it's 1-skeleton $\Sp(P) =: \Gha$--- by deck transformations. 
We know the quotient of a universal cover by the group of deck transformations is the space itself  \cite[p 70]{HA02}. Thus the quotient of $\widehat{\Gha}$ by $G_x$ is $\sC$, and so the quotient of $\Gha$ by $G_x$ is $C$. Since $C$ is finite when $X$ is, we deduce that the action of $G_x$ on $\Gha$ is co-compact.

Lastly, we claim that the action of $G_x$ by deck transformations on $\Gha$ is properly discontinuous. Any compact subset $K \subset \Gha$ is bounded in the graph metric. By \Prr{genloctrans}, $G_x$ acts regularly on $V_x$, and in particular the stabiliser of each vertex is trivial. Our claim now easily follows,  e.g.\ by using the fact that every cellular action on a CW-complex with finite stabilisers of cells is properly discontinuous \cite[Theorem~9,~(2)=(10)]{KapPD}. 

To summarise, the action of $G_x$ on $\Gha$ is  properly discontinuous and co-compact. The \v{S}varc--Milnor lemma \cite{Mi68} says exactly that $G_x$ is finitely generated, and quasi-isometric to $\Gha$ for any such action.
\end{proof}

\section{Line graphs of Cayley graphs admit \pp s} \label{sec Line}

In this section we show that every line graph of a Cayley graph can be represented as a \pcg. For this we will use 1- and 2-factorisations of the complete graphs as a tool.
Let $K_n$ be the complete graph on $n$-vertices. If $n$ is odd then $K_n$ has a Hamiltonian decomposition, a partition of the edges into spanning cycles \cite{AB08}. If $n$ is even, then a special case of Baranyai's theorem \cite{BZ75} gives us a 1-factorisation of $K_n$, i.e.\ a partition of the edges into perfect matchings. 

Thus in either case, we have found a \pf\ multicycle colouring $c: E(K_n) \rightarrow \Omega$ of $K_n$. Next, we want to associate each colour $\omega \in \Omega$ with a permutation $\pi_\omega \in Sym_n$ of the vertices of $K_n$. To do so, for each $\omega \in \Omega$ such that $c^{-1}(\omega)$ is 2-regular, we pick an orientation $O_\omega \subset c^{-1}(\omega)$, (such that $O_{\omega} \cap O_{\omega}^{-1} = \emptyset$ and $O_{\omega} \cup O_{\omega}^{-1} = c^{-1}(\omega)$), and let $\pi_\omega$ be the corresponding permutation (sending each vertex to its successor in $O_{\omega}$. For each $\omega \in \Omega$ such that $c^{-1}(\omega)$ is 1-regular, we let $\pi_\omega$ be the permutation that exchanges the two end-vertices of each edge in $c^{-1}(\omega)$.

\comment{
\begin{lem} \label{complete-graph-mutlicycle-colouring}
	\cite{AB08,BZ75} The complete graph $K_n$ has multicycle colouring $c: E(K_n) \rightarrow \Omega$, which can be represented as permutations $c' : \Omega \rightarrow Sym_n$.
\end{lem}
\begin{proof}
}

\begin{prop}
Let $\Gamma=\Cay\langle \sS \vert \sR \rangle$ be a Cayley graph. Then the line graph $L(\Gamma)$ can be represented as $\Sp(P)$ for a \pp\ $P$ with at most $|\sS|$ vertex classes.
\end{prop}
\begin{proof}
The \pp\ $P$ we will construct will have one vertex class for each generator in $\sS$. Since the edges of $L(\Gamma)$ are precisely the pairs of incident edges of $\Gamma$, we will identify the generators of $P$ with pairs of generators  $s,t\in \sS$. Since we need to pay attention to the directions of the edges of $\Gamma$, each such pair $s,t$ will give rise to four generators of $P$, indexed by the elements of $\{-1,1\}^2$. Similarly, each $s\in \sS$ will  give rise to two generators of $P$, since there are pairs of incident edges of $\Gamma$ labelled by $s$, and there are two choices for their directions. The relators of $P$ will be of two kinds. The first kind is just obtained by rewriting the elements of $\sR$ in terms of the new generators. The second kind will correspond to closed walks in $L(\Gamma)$ contained in the star of a vertex of $\Gamma$.

We proceed with the formal definition of $P$. The vertex classes of $P$ will be identified with the generating set $\sS$ of $\Gamma$.    Let $K_{\sS}$ denote the complete graph with $V(K_{\sS}) = \sS$. From the above discussion we obtain a multicycle colouring  $M \subset \Sym_{\sS}$ of $K_{\sS}$ where each colour is identified with a permutation of $\sS$. The generating set of our \pp\ $P$ comprises the formal symbols\\ $\sU = \{ e, e^{-1} \} \cup \{m_{i,j} \vert m \in M, i,j \in \{-1,1\} \mbox{ where } m^2 \not = 1\}$ and\\ $\sI = \{m_{i,j} \vert m \in M, i,j \in \{-1,1\} \mbox{ where } m^2 = 1\}$.  Set $\sS' = \sU \cup \sI$, the generators of $P$. We need to associate a permutation $\phi(s)$ of the vertex classes with each  $s\in \sS'$, and we do so by
\[
\phi: \begin{array}{c} m_i\\ e,e^{-1} \end{array} \mapsto \begin{array}{c}
m\\ 1_\sS
\end{array}.
\]
Let $\theta : \overrightarrow{E}(K_{\sS}) \rightarrow M \cup M^{-1}$ be the colouring of $K_{\sS}$ by $M \cup M^{-1}$. We can think of $\theta$ as a map from $\sS \times \sS \backslash \{(s,s) \vert s \in \sS\}$  to $ M \cup M^{-1}$ where $\theta(a,b)(a) = b$. Let $S = \{s, s^{-1} \vert s \in \sS\}$ be $\sS$ with formal inverses. Define a map $\chi: S \times S \backslash \{(s,s^{-1}) \vert s \in \sS\} \rightarrow \sS'$ where
\[
\chi(a,b) = \begin{cases} e & \mbox{ if } a = b \in \sS\\
e^{-1} & \mbox{ if } a = b \notin \sS\\
m_{i,j} & \mbox{ if } \theta(a,b) = m \mbox{ where  } a^i,b^j \in \sS.  \end{cases}
\]
Here we make the identification that $(m_{i,j})^{-1} = (m^{-1})_{-j,-i}$. 

We now define the sets of relators $\sR_a, a\in \sS$ of $P$. For each relator $r := a_1 a_2 \ldots a_k \in \sR$ we add\\ $\chi(r) := \chi(a_1,a_2)\chi(a_2,a_3) \ldots \chi(a_{k-1},a_k) \chi(a_k,a_1)$ to $\sR_{a_1^{\pm 1}}$. (These are the relators of the first kind as explained at the beginning of the proof.) Lastly, we add relations (of the second kind) corresponding to the star of each vertex of $\Gamma$ as follows. Let $a_1 \ldots a_k \in W_{\sS}$ be any word equaling the identity in $MF_P$, and add $\chi(a_1,a_2) \ldots \chi(a_k,a_1)$ to $\sR_{a_1^{\pm 1}}$, where $\chi(s,s^{-1})$ is the empty word. Let $\sR':= \{\sR_a, a\in \sS\}$. We have now constructed our presentation $P := \langle \sS \vert \sU \vert \sI \vert \phi \vert \sR' \rangle$.

Next, we prove that $\Sp(P)$ is isomorphic to $L(\Gha)$. First label
\[
V(L(\Gha)) = \{ [(g,gs)] \vert g \in \Gpa, \ s \in \sS\} \mbox{ and } \overrightarrow{E}(L(\Gha)) = \{ (g,s_1,s_2) \vert g \in \Gpa, \ s_1,s_2 \in \sS \cup \sS^{-1}, \ s_1 \not = s_2^{-1} \}
\]
so that the edge $(g,s_1,s_2)$ connects $[(g,gs_1)]$ and $[(gs_1,gs_1s_2)]$. Let $C := C(P)$ be the presentation graph of $P$. Then we can define a map $\epsilon : L(\Gha) \rightarrow C$ by letting $\epsilon([(g,gs)]) = s$ and letting $\epsilon((g,s_1,s_2))$ be the edge of colour $\chi(s_1,s_2)$ coming from $s_1^{\pm 1} \in \sS$. One can show that the relations in $\sR_x$ hold in $L(\Gha)$ for all $x \in \sS$. It remains to show that these relations suffice.

Intuitively we are going to argue that any closed walk $p$ in $L(\Gha)$ is labelled by some $r \in \langle\langle\sR\rangle \rangle_{MF_P}$ interwoven with relations coming from the stars at the vertices. One can observe this by just projecting $p$ to a closed walk in $\Gha$, where after some cancelations happening within the stars of vertices, we are left with a closed walk labelled by a word $r$ than can be expressed in terms of the relators in $\sR$. We proceed with this formally.

Define a topological map $\Phi: L(\Gha) \rightarrow \Gha$ by mapping $[(g,gs)] \in V(L(\Gha))$ to the midpoint of the edge $(g,gs)$, and $(g,s_1,s_2) \in E(L(\Gha))$ to the arc in the star of $gs_1$ connecting the midpoints of $[(g,gs_1)]$ and $[(gs_1,gs_1s_2)]$. Consider a closed walk $p$ in $L(\Gha)$. We can write $p = \prod_{i=0}^{n-1} (g^i,s^i_1,s^i_2)$. As $\Phi(p)$ is a closed walk in $\Gha$ we know it can be contracted to a path given by $g_0 g_1 \ldots g_{m-1}$ for $g_i \in \Gpa$. Now we want to group the edges of $p$ by the stars of vertices of $\Gha$ they lie in. For this, we subdivide the interval $\{0, \ldots, n-1\}$ into disjoint subintervals $\{I_j\}_{j=0}^{m-1}$ such that $\Phi(g^i,s^i_1,s^i_2)$ lies in the star of $g_j$ for all $i \in I_j$ and  $0 \leq j \leq m-1$ (we can assume without loss of generality that no $I_j$ has to be the union of an initial and a final subinterval of $\{0, \ldots, n-1\}$ by rotating $p$ appropriately). Thus $p= \prod_{j = 0}^{m-1} \left (\prod_{i \in I_j} (g^i,s^i_1,s^i_2) \right )$.

To each $j$ we can also associate $s_j \in \sS \cup \sS^{-1}$ so that $g_js_j = g_{j+1}$; these are the generators that $p$ uses in order to move from one star to the next.\\

We modify $p$ into a closed walk $p'$ by inserting pairs of edges that have the same end-vertices and opposite directions each time that $p$ moves from one star to the next. More formally, we define 
\[
p' := \prod_{j = 0}^{m-1} \left( \left (\prod_{i \in I_j} (g^i,s^i_1,s^i_2) \right ) (g_{j+1}, s_j^{-1}, s_{j-1}^{-1}) (g_{j-1}, s_{j-1}, s_{j}) \right ). 
\] 
Notice that by contracting these pairs of opposite edges $(g_{j+1}, s_j^{-1}, s_{j-1}^{-1}) (g_{j-1}, s_{j-1}, s_{j})$ we obtain $p$. Moreover, the sub-walk
\[
\left ( \prod_{i \in I_j} (g^i,s^i_1,s^i_2) \right ) (g_{j+1}, s_j^{-1}, s_{j-1}^{-1})
\]
of $p'$ stays within the star of $g_j$ by definition, and it is a closed walk starting and ending at $[(g_{j-1},g_j)]$. Therefore, it is labelled by one of our relators of the second kind. Easily, $\Phi(p)$ is homotopic to  $\Phi(p')$. Moreover,
\begin{align*}
  \Phi(p') & = \prod_{j = 0}^{m-1} \Phi\left( \left ( \prod_{i \in I_j} (g^i,s^i_1,s^i_2) \right ) (g_{j+1}, s_j^{-1}, s_{j-1}^{-1}) \right ) \Phi(g_{j-1}, s_{j-1}, s_{j})\\ 
& = \prod_{j = 0}^{m-1} \Phi(g_{j-1}, s_{j-1}, s_{j})
\end{align*}
since a closed walk contained in a star is 0-homotopic. 
Now $\prod_{j = 0}^{m-1} (g_{j-1}, s_{j-1}, s_{j})$ is a closed walk in $L(\Gha)$ no three consequtive edges of which are contained in the star of a vertex of $\Gha$ because of the way we chose the $I_j$. This implies that the word $s_0 \ldots s_{m-1}$ labelling this walk is a relation of $\Gha$, and so it can be written as a product of conjugates of relators $\sR$. Recalling that each such relator was admitted as a relator (of the first kind) in $\sR'$, we conclude that the word labelling $p$ can be written as products of conjugates of words in $\sR'$.
\end{proof}

We explicate an example of this below.

\begin{examp}
Consider $D_{10} = \langle a,b \vert a^5, b^2, aba^{-1}b^{-1} \rangle$, which has the  Cayley graph and line graph thereof shown in Figure~\ref{dihedral-example}. 
\showFigTikz{
\begin{tikzpicture}[scale = 0.25]
\node (G-11) at (7.5,17) [circle,fill=black,scale = 0.3] {};
\node (G-12) at (15,12) [circle,fill=black,scale = 0.3] {};
\node (G-13) at (12,3) [circle,fill=black,scale = 0.3] {};
\node (G-14) at (3,3) [circle,fill=black,scale = 0.3] {};
\node (G-15) at (0,12) [circle,fill=black,scale = 0.3] {};
\node (G-21) at (7.5,12) [circle,fill=black,scale = 0.3] {};
\node (G-22) at (10,10) [circle,fill=black,scale = 0.3] {};
\node (G-23) at (9,7) [circle,fill=black,scale = 0.3] {};
\node (G-24) at (6,7) [circle,fill=black,scale = 0.3] {};
\node (G-25) at (5,10) [circle,fill=black,scale = 0.3] {};
\draw[->][red] (G-11) edge (G-12);
\draw[->][red] (G-12) edge (G-13);
\draw[->][red] (G-13) edge (G-14);
\draw[->][red] (G-14) edge (G-15);
\draw[->][red] (G-15) edge (G-11);
\draw[->][red] (G-21) edge (G-22);
\draw[->][red] (G-22) edge (G-23);
\draw[->][red] (G-23) edge (G-24);
\draw[->][red] (G-24) edge (G-25);
\draw[->][red] (G-25) edge (G-21);
\draw[<-][blue] (G-11) edge [bend left] (G-21);
\draw[->][blue] (G-11) edge [bend right] (G-21);
\draw[<-][blue] (G-12) edge [bend left] (G-22);
\draw[->][blue] (G-12) edge [bend right] (G-22);
\draw[<-][blue] (G-13) edge [bend left] (G-23);
\draw[->][blue] (G-13) edge [bend right] (G-23);
\draw[<-][blue] (G-14) edge [bend left] (G-24);
\draw[->][blue] (G-14) edge [bend right] (G-24);
\draw[<-][blue] (G-15) edge [bend left] (G-25);
\draw[->][blue] (G-15) edge [bend right] (G-25);
\end{tikzpicture} \hspace{0.5 in}
\begin{tikzpicture}[scale = 0.2]
\node (G-11) at (7.5,0) [circle,fill=red,scale = 0.3] {};
\node (G-12) at (0,7) [circle,fill=red,scale = 0.3] {};
\node (G-13) at (4,15) [circle,fill=red,scale = 0.3] {};
\node (G-14) at (11,15) [circle,fill=red,scale = 0.3] {};
\node (G-15) at (15,7) [circle,fill=red,scale = 0.3] {};
\node (G-21) at (7.5,5) [circle,fill=red,scale = 0.3] {};
\node (G-22) at (5,7) [circle,fill=red,scale = 0.3] {};
\node (G-23) at (6,9) [circle,fill=red,scale = 0.3] {};
\node (G-24) at (9,9) [circle,fill=red,scale = 0.3] {};
\node (G-25) at (10,7) [circle,fill=red,scale = 0.3] {};
\node (G-31) at (5.5,4) [circle,fill=blue,scale = 0.3] {};
\node (G-32) at (3.7,5.5) [circle,fill=blue,scale = 0.3] {};
\node (G-41) at (3,8.5) [circle,fill=blue,scale = 0.3] {};
\node (G-42) at (4,10.5) [circle,fill=blue,scale = 0.3] {};
\node (G-51) at (6.4,12) [circle,fill=blue,scale = 0.3] {};
\node (G-52) at (8.6,12) [circle,fill=blue,scale = 0.3] {};
\node (G-61) at (11,10.5) [circle,fill=blue,scale = 0.3] {};
\node (G-62) at (12,8.5) [circle,fill=blue,scale = 0.3] {};
\node (G-71) at (11.5,5.5) [circle,fill=blue,scale = 0.3] {};
\node (G-72) at (9.7,4) [circle,fill=blue,scale = 0.3] {};
\draw[->][orange] (G-11) edge [bend left] (G-12);
\draw[->][orange] (G-12) edge [bend left] (G-13);
\draw[->][orange] (G-13) edge [bend left] (G-14);
\draw[->][orange] (G-14) edge [bend left] (G-15);
\draw[->][orange] (G-15) edge [bend left] (G-11);
\draw[->][orange] (G-21) edge [bend left] (G-22);
\draw[->][orange] (G-22) edge [bend left] (G-23);
\draw[->][orange] (G-23) edge [bend left] (G-24);
\draw[->][orange] (G-24) edge [bend left] (G-25);
\draw[->][orange] (G-25) edge [bend left] (G-21);
\draw[->][orange] (G-31) edge [bend left] (G-32);
\draw[->][orange] (G-32) edge [bend left] (G-31);
\draw[->][orange] (G-41) edge [bend left] (G-42);
\draw[->][orange] (G-42) edge [bend left] (G-41);
\draw[->][orange] (G-51) edge [bend left] (G-52);
\draw[->][orange] (G-52) edge [bend left] (G-51);
\draw[->][orange] (G-61) edge [bend left] (G-62);
\draw[->][orange] (G-62) edge [bend left] (G-61);
\draw[->][orange] (G-71) edge [bend left] (G-72);
\draw[->][orange] (G-72) edge [bend left] (G-71);
\draw[<->,dashed][green] (G-31) edge (G-21);
\draw[<->,dashed][purple] (G-31) edge (G-22);
\draw[<->][green] (G-31) edge (G-11);
\draw[<->][purple] (G-31) edge (G-12);
\draw[<->,dashed][green] (G-32) edge (G-11);
\draw[<->,dashed][purple] (G-32) edge (G-12);
\draw[<->][green] (G-32) edge (G-21);
\draw[<->][purple] (G-32) edge (G-22);
\draw[<->,dashed][green] (G-41) edge (G-22);
\draw[<->,dashed][purple] (G-41) edge (G-23);
\draw[<->][green] (G-41) edge (G-12);
\draw[<->][purple] (G-41) edge (G-13);
\draw[<->,dashed][green] (G-42) edge (G-12);
\draw[<->,dashed][purple] (G-42) edge (G-13);
\draw[<->][green] (G-42) edge (G-22);
\draw[<->][purple] (G-42) edge (G-23);
\draw[<->,dashed][green] (G-51) edge (G-23);
\draw[<->,dashed][purple] (G-51) edge (G-24);
\draw[<->][green] (G-51) edge (G-13);
\draw[<->][purple] (G-51) edge (G-14);
\draw[<->,dashed][green] (G-52) edge (G-13);
\draw[<->,dashed][purple] (G-52) edge (G-14);
\draw[<->][green] (G-52) edge (G-23);
\draw[<->][purple] (G-52) edge (G-24);
\draw[<->,dashed][green] (G-61) edge (G-24);
\draw[<->,dashed][purple] (G-61) edge (G-25);
\draw[<->][green] (G-61) edge (G-14);
\draw[<->][purple] (G-61) edge (G-15);
\draw[<->,dashed][green] (G-62) edge (G-14);
\draw[<->,dashed][purple] (G-62) edge (G-15);
\draw[<->][green] (G-62) edge (G-24);
\draw[<->][purple] (G-62) edge (G-25);
\draw[<->,dashed][green] (G-71) edge (G-25);
\draw[<->,dashed][purple] (G-71) edge (G-21);
\draw[<->][green] (G-71) edge (G-15);
\draw[<->][purple] (G-71) edge (G-11);
\draw[<->,dashed][green] (G-72) edge (G-15);
\draw[<->,dashed][purple] (G-72) edge (G-11);
\draw[<->][green] (G-72) edge (G-25);
\draw[<->][purple] (G-72) edge (G-21);
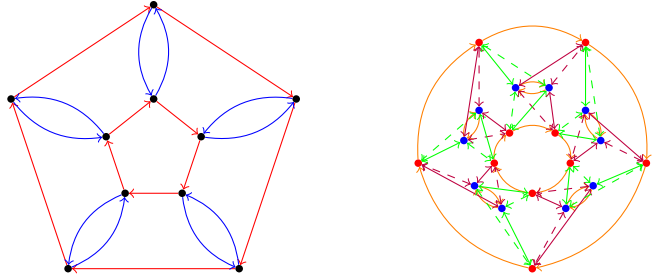
\end{tikzpicture}}{dihedral-example}{$\Cay\langle a,b \vert a^5, b^2, aba^{-1}b^{-1} \rangle$ and its line graph.}
As $K_{\{a,b\}}$ is a single edge, we have $M = \{(1,2)\} =: \{m\}$ with $m_{i,j} \mapsto (1,2)$ for $i,j \in \{1,-1\}$ and $e, e^{-1} \mapsto (1)(2)$ as generators. Define the following function
\[
\chi:
\begin{array}{cccccccccccccc}
aa & \rightarrow & e & & & ab & \rightarrow & m_{1,1} & & & b^{-1}a^{-1} & \rightarrow & m_{-1,-1}\\ 
a^{-1}a^{-1} & \rightarrow & e^{-1} & & & ab^{-1} & \rightarrow & m_{1,-1} & & & ba^{-1} & \rightarrow & m_{1,-1}\\
bb & \rightarrow & e & & & a^{-1}b & \rightarrow & m_{-1,1} & & & b^{-1}a & \rightarrow & m_{-1,1}\\
b^{-1}b^{-1} & \rightarrow & e^{-1} & & & a^{-1}b^{-1} & \rightarrow & m_{-1,-1} & & & ba & \rightarrow & m_{1,1}.\\
\end{array}
\]
The original relators $a^5, b^2, aba^{-1}b^{-1}$ are thus translated into relators in the resulting \pp\ as follows: $a^5 \rightarrow e^5 \in \sR_a$, $b^2 \rightarrow e^2 \in \sR_b$ and $aba^{-1}b^{-1} \rightarrow m_{1,1}m_{1,-1}m_{-1,-1}m_{-1,1} \in \sR_a$. Lastly, we add relations of the second kind shown in Figure~\ref{star-relations}, 
\showFigTikz{
\begin{tikzpicture}[scale = 0.2]
\node (M) at (0,0) [circle,fill=black,scale = 0.5] {};
\node (oa) at (-3,3) [circle,fill=white,scale = 0.1] {};
\node (ia) at (3,3) [circle,fill=white,scale = 0.1] {};
\node (ob) at (-3,-3) [circle,fill=white,scale = 0.1] {};
\node (ib) at (3,-3) [circle,fill=white,scale = 0.1] {};
\draw[->][red] (ia) edge (M);
\draw[->][red] (M) edge (oa);
\draw[->][blue] (ib) edge (M);
\draw[->][blue] (M) edge (ob);
\draw[->] (ia) edge [bend left] (oa);
\draw[->] (oa) edge [bend left] (ob);
\draw[->] (ob) edge [bend right] (ia);
\end{tikzpicture} \hspace{1.5 in}
\begin{tikzpicture}[scale = 0.2]
\node (M) at (0,0) [circle,fill=black,scale = 0.5] {};
\node (oa) at (-3,3) [circle,fill=white,scale = 0.1] {};
\node (ia) at (3,3) [circle,fill=white,scale = 0.1] {};
\node (ob) at (-3,-3) [circle,fill=white,scale = 0.1] {};
\node (ib) at (3,-3) [circle,fill=white,scale = 0.1] {};
\draw[->][red] (ia) edge (M);
\draw[->][red] (M) edge (oa);
\draw[->][blue] (ib) edge (M);
\draw[->][blue] (M) edge (ob);
\draw[->] (ia) edge [bend left] (oa);
\draw[->] (oa) edge [bend right] (ib);
\draw[->] (ib) edge [bend left] (ia);
\end{tikzpicture}\\
$aaa^{-1}bb^{-1}a^{-1} \rightarrow em_{-1,1}m_{-1,-1}$ \hspace{0.3 in} $aaa^{-1}b^{-1}ba^{-1} \rightarrow em_{-1,-1}m_{1,-1}$\\
\begin{tikzpicture}[scale = 0.2]
\node (M) at (0,0) [circle,fill=black,scale = 0.5] {};
\node (oa) at (-3,3) [circle,fill=white,scale = 0.1] {};
\node (ia) at (3,3) [circle,fill=white,scale = 0.1] {};
\node (ob) at (-3,-3) [circle,fill=white,scale = 0.1] {};
\node (ib) at (3,-3) [circle,fill=white,scale = 0.1] {};
\draw[->][red] (ia) edge (M);
\draw[->][red] (M) edge (oa);
\draw[->][blue] (ib) edge (M);
\draw[->][blue] (M) edge (ob);
\draw[->] (ia) edge [bend right] (ob);
\draw[->] (ob) edge [bend left] (ib);
\draw[->] (ib) edge [bend left] (ia);
\end{tikzpicture} \hspace{1.5 in}
\begin{tikzpicture}[scale = 0.2]
\node (M) at (0,0) [circle,fill=black,scale = 0.5] {};
\node (oa) at (-3,3) [circle,fill=white,scale = 0.1] {};
\node (ia) at (3,3) [circle,fill=white,scale = 0.1] {};
\node (ob) at (-3,-3) [circle,fill=white,scale = 0.1] {};
\node (ib) at (3,-3) [circle,fill=white,scale = 0.1] {};
\draw[->][red] (ia) edge (M);
\draw[->][red] (M) edge (oa);
\draw[->][blue] (ib) edge (M);
\draw[->][blue] (M) edge (ob);
\draw[->] (oa) edge [bend left] (ob);
\draw[->] (ob) edge [bend left] (ib);
\draw[->] (ib) edge [bend right] (oa);
\end{tikzpicture}\\
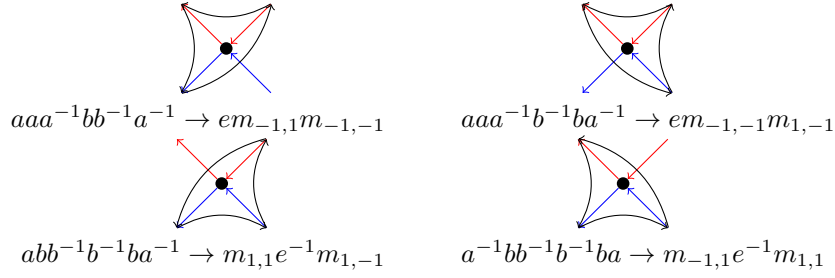
$abb^{-1}b^{-1}ba^{-1} \rightarrow m_{1,1}e^{-1}m_{1,-1}$ \hspace{0.3 in} $a^{-1}bb^{-1}b^{-1}ba \rightarrow m_{-1,1}e^{-1}m_{1,1}$}{star-relations}{Example of relations of the second kind}
which are enough to generate the rest of the relations. The resulting \pp\ is
\[
\langle a,b \vert \begin{array}{c}\sI =  m_{1,1}, m_{1,-1}, m_{-1,1}, m_{-1,-1}\\ \sU = e, e^{-1} \end{array} \rightarrow \begin{array}{c} (12)\\ (1)(2) \end{array} \vert \begin{array}{c} \{e^5, m_{1,1}m_{1,-1}m_{-1,-1}m_{-1,1}, em_{-1,1}m_{-1,-1},\\  em_{-1,-1}m_{1,-1}, m_{1,1}e^{-1}m_{1,-1}, m_{-1,1}e^{-1}m_{1,1}\}, \{e^2\} \end{array} \rangle.
\]
\end{examp}
\section{A Cubic 2-ended vertex transitive graph which is not Cayley} \label{sec 2-ended}

In this section we construct an example of a 2-ended cubic vertex transitive graph which is not a Cayley graph. This answers a question of Watkins \cite{Wa90} also appearing in \cite{GL16}.

\medskip
Let $\Gha$ be the graph with $V(\Gha) = \{v_{n,k} \vert n \in \bZ, k \in \bZ/10\bZ\}$ and
\[
 E(\Gha) = \{v_{n,k} v_{n,k+1}\ \vert \ n \in \bZ, \ k \in \bZ/10\bZ\} \cup \{ v_{n,2k+1} v_{n+1,4k+2} \ \vert \ n \in \bZ, \ k \in \bZ/10\bZ\}.
\]
By construction, $\Gha$ is a cubic graph. A useful way to think of this graph is as a 2-way infinite stack of layers  $L_n := \{v_{n,k} \vert k \in \bZ/10\bZ\}$. Each $L_n$ spans a 10-cycle, and between any two layers $L_n$ and $L_{n+1}$ there is a Petersen-graph-like structure. 

\begin{cla}
	$\Gha$ is 2-ended.	
\end{cla}
\begin{proof}
We will show that $\Gha$ is quasi-isometric to $\Ghb := \Cay(\bZ, \{1\})$, hence 2-ended. Let $d_{\Gha}$ and $d_{\Ghb}$ be the path metric in $\Gha$ and $\Ghb$ respectively. Our  quasi-isometry is the map
$f: \Gha \rightarrow \Ghb$ defined by  $f(v_{n,k}) := n$.

It is straightforward to check that $d_{\Gha}(v_{n,0}, v_{n+1,0}) = 2$ and $d_{\Gha}(v_{n,k},v_{n,0}) \leq k$  by the definition. Now for any two vertices $v_{n,k}, v_{n',k'} \in V(\Gha)$ where $n \leq n'$, we have
\begin{align*}
d_{\Gha}(v_{n,k},v_{n',k'}) \leq & d_{\Gha}(v_{n,k},v_{n,0}) + \sum_{n \leq i < n'} d_{\Gha}(v_{i,0},v_{i+1,0}) + d_{\Gha}(v_{n',0},v_{n',k}) & \mbox{by the triangle inequality}\\
\leq & 2 (n-n') + (k + k') & \mbox{by the two facts above}\\
\leq & 2 d_{\Ghb}(f(v_{n,k}), f(v_{n',k'})) + 20 & \mbox{by the definition of }f.
\end{align*} 
Another straightforward consequence of the definition of $\Gha$ is that $d_{\Gha}(v_{n,k}, v_{n',k'}) \geq n - n'$, which combined with the above inequality yields
\[
d_{\Gha}(v_{n,k},v_{n',k'})/2 - 20 \leq d_{\Ghb}(f(v_{n,k}), f(v_{n',k'})) = n - n' \leq d_{\Gha}(v_{n,k}, v_{n',k'}).
\]
As $f$ is surjective, this means that it is a quasi-isometry. Since the number of ends of a graph is invariant under quasi-isometry, $\Gha$ has 2-ends as  $\bZ$ does.
\end{proof}

\begin{cla}
	$\Gha$ is vertex transitive.
\end{cla}
\begin{proof}

We introduce the following two maps $\sigma, \tau : V(\Gha) \to V(\Gha)$, which will allow us to map any vertex of $\Gha$ to any other:
\[
\sigma(v_{n,k}) = v_{n+1,k} \ \mbox{ and } \ \tau(v_{n,k}) = \begin{cases} v_{-n, k+1} & \mbox{ if } n \equiv 0 \mbox{ (mod }4\mbox{)}\\ v_{-n, 3-k} & \mbox{ if } n \equiv 1 \mbox{ (mod }4\mbox{)}\\v_{-n, k+9} & \mbox{ if } n \equiv 2 \mbox{ (mod }4\mbox{)}\\v_{-n, 7-k} & \mbox{ if } n \equiv 3 \mbox{ (mod }4\mbox{)} \end{cases}.
\] 
Intuitively $\sigma$ just shifts all the layers up by 1, so $L_n$ is mapped to $L_{n+1}$ keeping the positions and orientations of the 10-cycles the same. Whereas $\tau$ rotates the 10-cycle on $L_0$ by one position, which flips the stack of layers by mapping $L_n$ to $L_{-n}$, and inverts the orientation of the 10-cycles at odd numbered layers.  
It is easy to see that $\tau$ and $\sigma$ preserve edges of the form $(v_{n,k},v_{n,k+1})$. Moreover, it is easy to see that $\sigma$ preserves edges of the form $(v_{n,2k+1}, v_{n+1,4k+2})$, we now check that $\tau$ also preserves these edges:
\begin{align*}
\tau : (v_{n,2k+1}, & v_{n+1,4k+2}) \mapsto \\
&\begin{array}{cccc} (v_{-n,2k+2},v_{-(n+1), 1 - 4k}) & = & (v_{-(n+1) + 1,4(-2k) + 2},v_{-(n+1), 2(-2k) + 1}) &  \mbox{ if } n \equiv 0 \mbox{ (mod }4\mbox{)}\\
(v_{-n,2 - 2k},v_{-(n+1), 4k + 1}) & = & (v_{-(n+1) + 1,4(2k) + 2},v_{-(n+1), 2(2k) + 1}) &  \mbox{ if } n \equiv 1 \mbox{ (mod }4\mbox{)}\\
(v_{-n,2k},v_{-(n+1), 5 - 4k}) & = & (v_{-(n+1) + 1,4(2-2k) + 2},v_{-(n+1), 2(2-2k) + 1}) &  \mbox{ if } n \equiv 2 \mbox{ (mod }4\mbox{)}\\
(v_{-n,6 - 2k},v_{-(n+1), 4k + 3}) & = & (v_{-(n+1) + 1,4(2k + 1) + 2},v_{-(n+1), 2(2k + 1) + 1}) &  \mbox{ if } n \equiv 3 \mbox{ (mod }4\mbox{)},\end{array}
\end{align*}

where we used that $8 \equiv -2$ (mod $10$). By changing the order of the vertices in the right hand side we see that these are indeed edges of $\Gha$ of the form $(v_{n,2k+1}, v_{n+1,4k+2})$. Thus we have checked that $\sigma, \tau \in \Aut(\Gha)$. Let 
$\Gpa := \langle \sigma, \tau \rangle \leq Aut(\Gha)$
be the group of automorphisms of $\Gha$ generated by  $ \sigma$ and $\tau$.
For any two vertices $v_{n,k}, v_{n',k'} \in V(\Gha)$, we have 
\begin{align*}
\sigma^{n'}\tau^{k'-k}\sigma^{-n}(v_{n,k}) = & \sigma^{n'}\tau^{k'-k}(v_{0,k})\\
= & \sigma^{n'}(v_{0,k'})\\
= & v_{n',k'}
\end{align*}
and so $\Gpa$ acts transitively on $V(\Gha)$. This proves Claim~2.
\end{proof}

\begin{cla} \label{not-regular}
	$\Gpa$ is not a regular subgroup of $\Aut(\Gha)$.
\end{cla}
\begin{proof}

Observe
\begin{align*}
\tau^{-3}\sigma\tau\sigma (v_{0,k}) & = \tau^{-3} \sigma \tau (v_{1,k})\\
& = \tau^{-3} \sigma (v_{-1, 3 - k})\\ & = \tau^{-3} (v_{0,3 - k})\\ & = v_{0,- k}.
\end{align*}
This implies $\tau^{-3}\sigma\tau\sigma(v_{0,0}) = v_{0,0}$, yet $\tau^{-3}\sigma\tau\sigma \not = 1_{\Gha}$ as $\tau^{-3}\sigma\tau\sigma(v_{0,1}) = v_{0,9}$. Thus the action of $G$ on $\Gha$ is not free, proving Claim~3.
\end{proof}

\medskip
We remind the reader the $n$th layer is denoted $L_n = \{v_{n,k} \vert \ k \in \bZ/10\bZ \}$ for $n \in \bZ$ and we define the partition $\sC := \{L_n\ \vert\ n \in \bZ\}$ of $V(\Gha)$.

\begin{cla} \label{fix-one-cycle}
	Let $\phi \in \Aut(\Gha)$ satisfy $\phi(L_a) = L_b$. If for some $\chi \in \Aut(\Gha)$ we have $\phi(x)=\chi(x)$ for every $x\in L_a$,  then  $\phi=\chi$. Moreover, $\phi$ preserves the partition $\sC := \{L_n\ \vert\ n \in \bZ\}$.
\end{cla}
\begin{proof}

Observe that $\phi(v_{a+1,2k})$ and $\phi(v_{a-1,2k+1})$ are uniquely determined by $\phi(x), x\in L_a$ because each vertex in $L_b$ has exactly one neighbour outside $L_b$. This in turn uniquely determines $\phi(v_{a+1,2k+1})$ and $\phi(v_{a-1,2k})$ by a similar argument. Continuing like this, we see that $\phi(\{v_{a + \epsilon,k} \vert k \in \bZ/10\bZ\}) = \{v_{b \pm \epsilon,k} \vert k \in \bZ/10\bZ\}$ for $\epsilon \in \{-1,1\}$. By an inductive argument, this uniquely determines $\phi$, and moreover $\phi$ preserves $\sC$.
\end{proof}

\begin{cla} \label{cycle-pres}
	Any $\phi \in \Aut(\Gha)$ preserves the partition $\sC = \{L_n \vert n \in \bZ\}$.
\end{cla}
\begin{proof}
 
Suppose $\phi$ does not fix $\sC$. By Claim \ref{fix-one-cycle} we have $\phi(L_0) \not = L_a$ for every $a \in \bZ$. Let
\[
n := \min\{n' \in \bZ \vert v_{n',k} \in \phi(L_0) \}.
\]
Thus there exist $l,k \in \bZ/10\bZ$ such that 
\[
\phi : \begin{array}{c} v_{0,l}\\ v_{0,l\pm1}\\ v_{0,l\pm2}\\ v_{0,l \pm 3} \end{array} \mapsto \begin{array}{c} v_{n,2k-1}\\ v_{n, 2k}\\ v_{n, 2k+1}\\ v_{n+1,4k+2} \end{array}.
\] 
Let $N(v_{n,k})$ denote the neigbourhood of $v_{n,k}$ in $\Gha$.  
Let $v_{\epsilon,a} \in N(v_{0,l\pm1})$ and $v_{-\epsilon,b} \in N(v_{0,l\pm2})$ for $\epsilon \in \{-1,1\}$. Then 
\begin{align*}
\phi(v_{\epsilon,a}) &\in N(\phi(v_{0,l\pm1})) \backslash \{\phi(v_{0,l}),\phi(v_{0,l\pm2})\}\\ 
& = \{v_{n-1,a'}\}\\ 
\phi(v_{-\epsilon,b}) &\in N(\phi(v_{0,l\pm2})) \backslash \{\phi(v_{0,l\pm1}),\phi(v_{0,l\pm3})\}\\ 
& = \{v_{n,2k+2}\}.
\end{align*}
Note that $\phi(v_{\epsilon,a}) = v_{n-1,a'}$ and $\phi(v_{-\epsilon,b}) = v_{n,2k+2}$ lie in the same connected component of $\Gha \backslash \phi(L_0)$ by the definition of $n$. However, $v_{\epsilon,a}$ and $v_{-\epsilon,b}$ lie in different connected components of $\Gha \backslash L_0$. This contradicts that $\phi$ is an automorphism of $\Gha$, and so our claim is proved.
\end{proof}

\begin{cla} \label{unique-determined}
	Any $\phi \in \Aut(\Gha)$ is uniquely determined by $\phi(v_{0,1})$ and $\phi(v_{0,2})$.
\end{cla}
\begin{proof}

Assume $\phi(v_{0,1}) = v_{a,b}$. Then by Claim \ref{cycle-pres}, $\phi(v_{0,2}) \in N(v_{a,b}) \cap \{v_{a,k} \vert k \in \bZ/10\bZ\} = \{v_{a,b+1},v_{a,b-1}\}$. In either case, by Claim \ref{cycle-pres} this uniquely determines $\phi(v_{0,k})$ for $k \in \bZ/10\bZ$. By Claim \ref{fix-one-cycle}, this uniquely determines $\phi$.
\end{proof}

We remark that this implies $\mbox{Stab}_{\Aut(\Gha)}(v_{0,0}) = \langle \tau^{-3}\sigma\tau\sigma \rangle = \bZ/2\bZ$. So $\langle \tau, \sigma \rangle = Aut(\Gamma)$.

\begin{cla}
	$\Gha$ is not a Cayley graph.
\end{cla}
\begin{proof}

Let $T \leq \Aut(\Gha)$ be a transitive subgroup. Thus we can find automorphisms $\sigma', \tau' \in T$ such that $\sigma'(v_{0,1}) = v_{1,1}$ and $\tau'(v_{0,1}) = v_{0,2}$. By Claim \ref{cycle-pres}, $T$ preserves the partition $\sC$. So either $\tau'(v_{0,2}) = v_{0,3}$ and $\tau' = \tau$ or $\tau'(v_{0,2}) = v_{0,1}$ and  
\[
\tau'(v_{n,k}) = \sigma \tau \sigma(v_{n,k}) = \tilde{\tau}(v_{n,k}) := \begin{cases} v_{-n, 3 - k} & \mbox{ if } n \equiv 0 \mbox{ (mod }4\mbox{)}\\ v_{-n, k-1} & \mbox{ if } n \equiv 1 \mbox{ (mod }4\mbox{)}\\v_{-n, 7-k} & \mbox{ if } n \equiv 2 \mbox{ (mod }4\mbox{)}\\v_{-n, k+1} & \mbox{ if } n \equiv 3 \mbox{ (mod }4\mbox{)} \end{cases}
\] 
by Claim \ref{unique-determined}. Similarly, either $\sigma'(v_{0,2}) = v_{1,2}$ and $\sigma' = \sigma$ or $\sigma'(v_{0,2}) = v_{1,0}$ and
\[
\sigma'(v_{n,k}) = \sigma \tau^{-1} \sigma \tau \sigma(v_{n,k}) = \tilde{\sigma}(v_{n,k}) := \begin{cases} v_{n+1, 2-k} & \mbox{ if } n \equiv 0 \mbox{ (mod }4\mbox{)}\\ v_{n+1, 4-k} & \mbox{ if } n \equiv 1 \mbox{ (mod }4\mbox{)}\\v_{n+1, 8-k} & \mbox{ if } n \equiv 2 \mbox{ (mod }4\mbox{)}\\v_{n+1, 6-k} & \mbox{ if } n \equiv 3 \mbox{ (mod }4\mbox{)} \end{cases}.
\] 
by Claim \ref{unique-determined}. By Claim \ref{not-regular}, if $\{\sigma', \tau'\} = \{\sigma, \tau\}$ then $T$ is not regular. If $\{\sigma', \tau'\} = \{\tilde{\sigma}, \tau\}$ then
\begin{align*}
\tau\tilde{\sigma}\tau\tilde{\sigma} (v_{0,k}) & = \tau \tilde{\sigma} \tau (v_{1,2-k})\\
& = \tau \tilde{\sigma} (v_{-1, k + 1})\\ 
& = \tau (v_{0,5 - k})\\ 
& = v_{0,6 - k}
\end{align*}
giving $\tau\tilde{\sigma}\tau\tilde{\sigma}(v_{0,3}) = v_{0,3}$ yet $\tau\tilde{\sigma}\tau\tilde{\sigma} \not = 1_{\Gha}$. Similarly, if $\{\sigma', \tau'\} = \{\sigma, \tilde{\tau}\}$ then
\begin{align*}
\tilde{\tau}\sigma\tilde{\tau}\sigma (v_{0,k}) & = \tilde{\tau}\sigma\tilde{\tau} (v_{1,k})\\
& = \tilde{\tau}\sigma (v_{-1, k-1})\\ 
& = \tilde{\tau} (v_{0,k-1})\\ 
& = v_{0,4 - k}
\end{align*}
giving $\tilde{\tau}\sigma\tilde{\tau}\sigma(v_{0,2}) = v_{0,2}$ yet $\tilde{\tau}\sigma\tilde{\tau}\sigma \not = 1_{\Gha}$. Lastly, if $\{\sigma', \tau'\} = \{\tilde{\sigma}, \tilde{\tau}\}$ then 
\begin{align*}
\tilde{\tau}\tilde{\sigma}\tilde{\tau}\tilde{\sigma} (v_{0,k}) & = \tilde{\tau} \tilde{\sigma} \tilde{\tau} (v_{1,2-k})\\
& = \tilde{\tau} \tilde{\sigma} (v_{-1, 1-k})\\ 
& = \tilde{\tau} (v_{0,5 + k})\\ 
& = v_{0,-2 - k}
\end{align*}
giving $\tilde{\tau}\tilde{\sigma}\tilde{\tau}\tilde{\sigma}(v_{0,9}) = v_{0,9}$ yet $\tilde{\tau}\tilde{\sigma}\tilde{\tau}\tilde{\sigma} \not = 1_{\Gha}$. Therefore $T$ is never regular, and so $\Gha$ is not a Cayley graph.
\end{proof}

\medskip
Combining the above claims we deduce 

\paragraph{\textbf{Theorem \ref{cubic-2-ended-vertex-transitive-example}}}
	$\Gha$ is a cubic 2-ended vertex transitive graph which is not a Cayley graph.\\
	\qed
	
During the course of this proof we have shown $Aut(\Gamma) = \langle \sigma, \tau \rangle$. Some further study, for which we thank Derek Holt, can show that \[ Aut(\Gamma) = \langle  \sigma, \tau \vert \tau^{10}, (\tau^{-1} \sigma \tau \sigma)^2, \sigma^{-1} \tau^2 \sigma\tau^{-4}, (\sigma^{-2}\tau)^2 \rangle. \]

\section{Conclusion} \label{sec Conc}

In this paper we showed that every vertex transitive graph admits a \pp, but we were not able to limit the number of vertex classes required.  This suggests

\begin{prb}{Can every  vertex transitive graph on at least 3 vertices be represented as a \pcg\ so that each vertex class contains at least two vertices?}
\end{prb}

Define the \defin{Cayleyness} of a (vertex transitive) graph $\Gha$ as the minimum number of vertex classes in any \pp\ of $\Gha$. Thus $\Gha$ is a \Cg\  if and only if it has Cayleyness 1.

\begin{prb}
Is there a vertex transitive graph of Cayleyness (at least) $n$ for every $n\in \mathbb{N}$? 
\end{prb}

Since the Cayleyness of a vertex-transitive graph $\Gamma$ divides $\vert V(\Gamma) \vert$, a potential approach to answering to this question is to enquire if for every prime $p \in \bN$, there is a vertex-transitive graph on $p^k$ vertices for some $k\in \bN$ that is not a \Cg.

We observe that the Diestel-Leader graph $DL(m,n)$ for $m\neq n$ has infinite Cayleyness. This follows by combining \Prr{quasi} with the fact that these graphs are not quasi-isometric to any finitely generated group \cite[Theorem~1.4]{EsFiWhCoa}. This motivates

\begin{prb}
Does a vertex-transitive graph $\Gamma$ have finite Cayleyness if and only if $\Gamma$ is quasi-isometric to a Cayley graph?
\end{prb}

\medskip
We say that a (vertex transitive) graph $\Gha$ is \defin{finitely presented} if it has a \pp\ with finitely many vertex classes and finitely many relators. Is this equivalent to $\pi_1(\Gha)$ being generated by walks of bounded length? It would be interesting to generalise results about finitely presented groups such as \cite{HigSub} to finitely presented graphs in our sense.

It is not hard to show, using group presentations, that there are finitely many finite extensions of any finitely presented group. When it comes to vertex transitive graphs the analogous question is still open and has been extensively studied, see \cite{GP95, Tr15, NT14} and references therein. We hope that \pp s will be useful in developing an analogous proof.

\section*{Acknowledgements}{We thank Derek Holt and Paul Martin for several corrections to an earlier draft of this work. We thank Wolfgang Woess for making us aware of Watkins' question.}


\bibliographystyle{alpha}
\bibliography{mycollection}

\begin{thebibliography}{FGW71}

\bibitem[Als08]{AB08}
B.~Alspach.
\newblock The wonderful {W}alecki construction.
\newblock {\em Bull. Inst. Combin. Appl.}, 52:7--20, 2008.

\bibitem[Bar75]{BZ75}
Z.~Baranyai.
\newblock On the factorization of the complete uniform hypergraph.
\newblock In {\em Infinite and finite sets ({C}olloq., {K}eszthely, 1973;
  dedicated to {P}. {E}rd\H{o}s on his 60th birthday), {V}ol. {I}}, volume~10,
  pages 91--108.Vol. 10. Colloq. Math. Soc., 1975.

\bibitem[Die05]{DiestelBook05}
R.~Diestel.
\newblock {\em Graph {T}heory \emph{(3rd edition)}}.
\newblock Springer-Verlag, 2005.
\newblock \\ Electronic edition available at:\\ {\small\tt
  http://www.math.uni-hamburg.de/home/diestel/books/graph.theory}.

\bibitem[DL01]{DiLeCon}
R.~Diestel and I.~Leader.
\newblock {A conjecture concerning a limit of non-{C}ayley graphs}.
\newblock {\em J.\ Algebraic Combinatorics}, 14:17--25, 2001.

\bibitem[EFW12]{EsFiWhCoa}
A.~Eskin, D.~Fisher, and K.~Whyte.
\newblock Coarse differentiation of quasi-isometries i: Spaces not
  quasi-isometric to cayley graphs.
\newblock {\em Annals of Mathematics}, 176(1):221--260, 2012.

\bibitem[FGW71]{FGW71}
R.~Frucht, J.~E. Graver, and M.~E. Watkins.
\newblock The groups of the generalized {P}etersen graphs.
\newblock {\em Proc. Cambridge Philos. Soc.}, 70:211--218, 1971.

\bibitem[Ger83]{Ge83}
S.~M. Gersten.
\newblock Intersections of finitely generated subgroups of free groups and
  resolutions of graphs.
\newblock {\em Invent. Math.}, 71(3):567--591, 1983.

\bibitem[GL20]{GL16}
G.~R. Grimmett and Z.~Li.
\newblock Cubic graphs and the golden mean.
\newblock {\em Discrete Math.}, 343(1):111638, 2020.

\bibitem[GP95]{GP95}
A.~Gardiner and C.~E. Praeger.
\newblock A geometrical approach to imprimitive graphs.
\newblock {\em Proc. London Math. Soc. (3)}, 71(3):524--546, 1995.

\bibitem[GR01]{GR01}
C.~{Godsil} and G.~{Royle}.
\newblock {\em Algebraic Graph Theory}, volume 207 of {\em Graduate Texts in
  Mathematics.}
\newblock volume 207 of Graduate Texts in Mathematics. Springer, 2001.

\bibitem[Hal35]{PH35}
P.~Hall.
\newblock On representatives of subsets.
\newblock {\em J. London Math. Soc.}, 10:26--30, 1935.

\bibitem[Hat02]{HA02}
A.~Hatcher.
\newblock {\em Algebraic topology}.
\newblock Cambridge University Press, Cambridge, 2002.

\bibitem[Hig61]{HigSub}
G.~Higman.
\newblock {Subgroups of Finitely Presented Groups}.
\newblock {\em Proceedings of the Royal Society of London. Series A,
  Mathematical and Physical Sciences}, 262(1311):455--475, 1961.

\bibitem[Kap]{KapPD}
M.~Kapovich.
\newblock A note on properly discontinuous actions.
\newblock ``https://www.math.ucdavis.edu/\textasciitilde
  kapovich/EPR/prop-disc.pdf".

\bibitem[Lee16]{Leemann}
P.~H. Leemann.
\newblock {\em On subgroups and Schreier graphs of finitely generated groups.
  PhD Thesis}.
\newblock Universit\'{e} de Gen\'{e}ve, 2016.

\bibitem[Lei83]{Le83}
F.~T. Leighton.
\newblock On the decomposition of vertex-transitive graphs into multicycles.
\newblock {\em J. Res. Nat. Bur. Standards}, 88(6):403--410, 1983.

\bibitem[Lov70]{LL70}
L.~Lov\'asz.
\newblock {\em Combinatorial structures and their applications: The
  factorization of graphs}, volume 1969 of {\em Proceedings of the Calgary
  International Conference on Combinatorial Structures and their Applications
  held at the University of Calgary, Calgary, Alberta, Canada, June}.
\newblock Gordon and Breach, Science Publishers, New York-London-Paris, 1970.

\bibitem[Mar81]{Ma81}
D.~Maru\v{s}i\v{c}.
\newblock On vertex symmetric digraphs.
\newblock {\em Discrete Math.}, 36(1):69--81, 1981.

\bibitem[Mil68]{Mi68}
J.~Milnor.
\newblock A note on curvature and fundamental group.
\newblock {\em J. Differential Geometry}, 2:1--7, 1968.

\bibitem[Mwa09]{Mw09}
E.~Mwambene.
\newblock Cayley graphs on left quasi-groups and groupoids representing
  {$k$}-generalised {P}etersen graphs.
\newblock {\em Discrete Math.}, 309(8):2544--2547, 2009.

\bibitem[NT14]{NT14}
E.~A. Neganova and V.~I. Trofimov.
\newblock Symmetric extensions of graphs.
\newblock {\em Izv. Ross. Akad. Nauk Ser. Mat.}, 78(4):175--206, 2014.

\bibitem[Pet91]{Petersen}
J.~Petersen.
\newblock {Die {T}heorie der regul\"{a}ren Graphen}.
\newblock {\em Acta Math.}, 15(1):193--220, 1891.

\bibitem[Tro15]{Tr15}
V.~I. Trofimov.
\newblock The finiteness of the number of symmetric extensions of a locally
  finite tree by means of a finite graph.
\newblock {\em Tr. Inst. Mat. Mekh.}, 21(3):303--308, 2015.

\bibitem[Wat90]{Wa90}
M.~E. Watkins.
\newblock Vertex-transitive graphs that are not {C}ayley graphs.
\newblock In {\em Cycles and rays ({M}ontreal, {PQ}, 1987)}, volume 301 of {\em
  NATO Adv. Sci. Inst. Ser. C Math. Phys. Sci.}, pages 243--256. Kluwer Acad.
  Publ., Dordrecht, 1990.

\end{thebibliography}

\newpage\appendix
\section{Appendix: Perfect matchings in infinite transitive graphs}
 \begin{center}
{\large authored by}

	\vskip 0.2 cm
	
	\noindent {\bf Matthias Hamann}\footnote{Supported by the Heisenberg-Programme of the Deutsche Forschungsgemeinschaft (DFG Grant HA8257/1-1).}  and  {\bf Alex Wendland}	\\

	\vskip 0.3 cm

Department of Mathematics, University of Warwick, Coventry, CV4 7AL, UK.
\end{center}	

\vskip 0.4 cm

In this Appendix we generalise \Tr{vertex-transitive-matching} of Godsil and Royle  to infinite graphs as follows. 

\paragraph{\textbf{Theorem \ref{inf-vertex-matching}}}
	Let $\Gha$ be a connected countably infinite vertex transitive graph. Then $\Gha$ has a perfect matching.\\

We say that a matching $M \subset E(\Gha)$ \emph{misses} $x \in V(\Gha)$ if $x$ is not in an edge of $M$. 
It is not hard to see that every connected, countable graph $\Gha$ admits a sequence of sets $B_n \subset V(\Gha)$, $n \in \bN$, such that
\begin{equation}\tag{$\dagger$}\label{B_n}
\vert B_n \vert < \infty,\ B_n \subset B_{n+1},\ \bigcup_n B_{n} = V(\Gha),\ \text{and the induced graph on }B_n\text{ is connected}.
\end{equation}

We define an order on the space of matchings for infinite graphs:

\begin{defn}
	Let $\Gha$ be a countably infinite graph, and let $M \subset E(\Gha)$ be a matching. The \defin{miss sequence} of $M$ with respect to a sequence $(B_n)_{n \in \bN}$ as in \eqref{B_n} is the sequence $(m_n)_{n \in \bN}$, where $m_n$ is the number of  vertices of $B_n$ that $M$ misses. For two matchings $M_1$ and $M_2$ with miss sequences with respect to $(B_n)_{n \in \bN}$ being $(a_n)$ and $(b_n)$ respectively, we write $M_1 < M_2$ if there exists $N \in \bN$ such that $a_n = b_n$ for all $n < N$ and $a_N > b_N$. We write $M_1\leq M_2$ if either $M_1<M_2$ or $a_n = b_n$ for all $n \in \bN$.  
\end{defn}

We now use Zorn's lemma to prove that there is a maximum matching with respect to this partial order.

\begin{lem}\label{max-matching}
	For every countable connected graph $\Gha$ and $(B_n)_{n \in \bN}$ satisfying {\rm (\ref{B_n})} there exists a maximal matching with respect to the miss sequence.
\end{lem}
\begin{proof}
	Assume there is a strictly increasing chain $(M_i)_{i\in\bN}$ of matchings in~$\Gha$ with respect to $<$. Let $M_i$ have miss sequence $(m^i_n)_{n \in \bN}$. We will inductively define infinite subsequences  $\sI_{n+1} \subset \sI_n \subset \bN$. Let $\sI_1 = \bN$. Assume we have defined $\sI_n$. As the induced graph on $B_n$ is finite there are finitely many matchings on it. Therefore infinitely many $i\in \sI_n$ induce the same matching $\sM_n$ on $B_n$. Let $\sI_{n+1}$ be such an infinite subset.
	
	Let $M = \bigcup_{i \in \bN} \sM_i$, which is a matching as $\sM_i \subset \sM_{i+1}$ for each $i \in \bN$ and every $\sM_i$ is a matching. Let $(m_n)_{n \in \bN}$ be the miss sequence of $M$. Then $m_n = m^i_n$ for $i \in \sI_n$ by the definition of~$M$. Let $i\in\bN$ and let $n$ be smallest such that $m_n^i\neq M_n^{i+1}$. As $(M_i)_{i\in\bN}$ is a chain, we have $m_n^i<m_n^{i+1}$. Let $j\in \sI_n$ with $j>i$. Then $M_i<M_j$. So we have $m_n^i<m_n^j=m_n$ and $m_k^i=m_k^j=m_k$ for $k<n$, and therefore $M_i<M$. Thus $M$ is an upper bound for the chain $(M_i)_{i \in \bN}$. Zorn's lemma now implies the existence of a maximum matching.
\end{proof}

We say that a path, ray, or cycle $P$ is \emph{alternating} with respect to  a matching $M$, if every other edge of $P$ is contained in $M$. The number of edges in a path is its \emph{length}. We remind the reader of the symmetric difference $S \oplus T = (S \cup T) \backslash (S \cap T)$.

\begin{lem} \label{matching-structure}
	Let $\Gha$ be a connected countably infinite graph and let $(B_n)_{n \in \bN}$ satisfy {\rm (\ref{B_n})}. If $M_1$ and $M_2$ are two maximum matchings, then the symmetric difference $M_1 \oplus M_2$ consists of even length cycles, double rays, and even length paths with end-vertices $u,v \in V(\Gha)$ such that $u,v \in B_{n} \backslash B_{n-1}$. Every component of $M_1 \oplus M_2$ is alternating with respect to $M_1$ and $M_2$.
\end{lem}
\begin{proof}
	As $M_1$ and $M_2$ are matchings, vertices in $M_1 \oplus M_2$ have degree $0,1$ or $2$. So $M_1 \oplus M_2$ consists of paths, cycles, rays and double rays, which are alternating with respect to $M_1$ and $M_2$. As the components are alternating with respect to $M_1$ all cycles must be of even length.
	
	Suppose $M_1 \oplus M_2$ contains a ray $R$ and let $y \in V(\Gha)$ be its starting vertex. Without loss of generality $M_1$ misses $y$. However $M_1 \oplus R$ does not miss $y$ and moreover it does not miss any vertex $M_1$ does not miss. This contradicts the maximality of $M_1$ as a matching. So $M_1 \oplus M_2$ does not contain any rays.
	
	Let $P$ be a path in $M_1 \oplus M_2$. If $P$ has odd length then the end-vertices $u,v \in V(\Gha)$ are without loss of generality missed by $M_1$. Therefore $M_1 \oplus P$ does not miss any vertex $M_1$ missed but in addition does not miss $u$ and $v$ which contradicts maximality of $M_1$. Therefore $P$ has even length.
	
	Let $u,v\in V(\Gha)$ be the end-vertices of~$P$. Suppose that $u \in B_N \backslash B_{N-1}$ and $v \in B_{N'} \backslash B_{N'-1}$ and suppose $N < N'$. Without loss of generality let $M_1$ miss $u$. Let $M_1$ and $M_1 \oplus P$ have miss sequences $(a_n)_{n \in \bN}$ and $(b_n)_{n \in \bN}$ with respect to $(B_n)_{n \in \bN}$ respectively. Then $a_n = b_n$ for $n < N$. However $M_1$ misses $u$ in $B(x,N)$ whereas $M_1 \oplus P$ does not. Therefore $a_{N} > b_{N}$. This contradicts the maximality of $M_1$, therefore $N = N'$.  
\end{proof}

\begin{lem} \label{missing-two}
	Let $\Gha$ be a connected countable graph and let $(B_n)_{n \in \bN}$ satisfy {\rm (\ref{B_n})}. Let $u,v \in V(\Gha)$ be such that no maximum matching misses both of them but such that there  are maximum matchings  $M_u$ and $M_v$ that miss $u$ and $v$ respectively. Then there is a path of even length in $M_u \oplus M_v$ with end-vertices $u$ and $v$.
\end{lem}
\begin{proof}
	In $M_u \oplus M_v$, both vertices $u$ and $v$ have degree 1. So $u$ and $v$ are the end points of paths in $M_u \oplus M_v$. If they are the end points of the same path we are done. So suppose not, and let $P$ be the path with end points $u$ and $y \not = v$. Lemma \ref{matching-structure} implies $u,y \in B_n \backslash B_{n-1}$. Then $M_v \oplus P$ is a matching that misses $u$ and $v$. As $u,y \in B_n \backslash B_{n-1}$ it has the same miss sequence as $M_v$ and thus is a maximum matching. This contradicts the assumption that no maximum matching misses $u$ and $v$. Thus $M_u \oplus M_v$ contains a path with end-vertices $u,v$.
\end{proof}

We call a vertex $v \in V(\Gha)$ \emph{critical}, if no maximum matching misses $v$.

\begin{lem} \label{path-critical-missing}
	Let $\Gha$ be a connected countable graph and let $(B_n)_{n \in \bN}$  satisfy {\rm (\ref{B_n})}. Let $u,v \in V(\Gha)$ be distinct and $P$ a path with end-vertices $u$ and $v$. If no vertex in $V(P) \backslash \{u,v\}$ is critical, then no maximum matching misses both $u$ and $v$.
\end{lem}
\begin{proof}
	We apply induction to the length of the path $P$. If $P$ is just an edge, then if a matching missed $u$ and $v$ the addition of this edge would increase the size of the matching.
	
	Assume the length of~$P$ is at least~$2$ and let $y \in V(P) \backslash \{u,v\}$. As $y$ is not critical there is matching $M_y$ missing $y$. By the induction hypothesis we know that no matching misses both $u$ and $y$ or both $v$ and $y$. Suppose $N$ was a maximum matching missing both $u$ and $v$. By Lemma \ref{missing-two}, we know that $N \oplus M_y$ contains a path with end-vertices $u$ and $y$ and a path with end-vertices $v$ and $y$. This contradicts $u \not = v$, so no such matching $N$ exists.
\end{proof}


The following is an easy consequence of a standard compactness argument. 

\begin{lem} \label{miss-none-in-balls}
	Let $\Gha$ be a connected countable graph and let $(B_n)_{n \in \bN}$ satisfy {\rm (\ref{B_n})}. If there exists a sequence of matchings $(M_i)_{i \in \bN}$ of $\Gha$ such that $M_n$ misses no vertex in $B_n$, then $\Gha$ has a matching missing no vertex.
\end{lem}
\begin{proof} 
	We inductively build subsequences $(M_i)_{i \in \sI_n}$ where $\sI_{n+1} \subset \sI_{n} \subset \bN$ are infinite subsets. Let $\sI_1 = \bN$. For $n \in \bN$ assume we have constructed $\sI_n$. Since $B_n$ is finite, there are only finitely many matchings in $B_n$. Therefore, there is an infinite sequence of matchings $(M_i)_{i \in \sI_n}$   all members of which induce the same matching $\sM_n$ on $B_n$. Let $\sI_{n+1}$ be the index set of that subsequence. As no $M_i$ with $i\geq n$ misses a vertex in $B_n$, the matching $\sM_n$ does not miss any vertex in $B_n$.
	
	Let $M = \cup_{n \in \bN} \sM_n$. As $\sM_n \subset \sM_{n+1}$ we have that $M$ is a matching in $\Gha$. By construction it covers all vertices of $\Gha$. 
\end{proof}

We now have all the tools we need to prove Theorem~\ref{inf-vertex-matching}.

\begin{proof}[Proof of Theorem~\ref{inf-vertex-matching}]
	Let $\Gha$ be a connected, countably infinite, vertex transitive graph. Let $(B_n)_{n \in \bN}$ satisfy {\rm (\ref{B_n})}. By vertex transitivity, either all vertices are critical or none are. If all are we are done, so suppose that no vertex of $\Gha$ is critical. By Lemma~\ref{path-critical-missing}, every maximum matching misses at most one vertex. Let $M$ be a maximal matching, which exists by Lemma~\ref{max-matching}, and $x \in V(G)$ be the vertex $M$ misses. As $\Gha$ is an infinite vertex transitive graph there exists an automorphism $\phi_n$ that maps $x$ to a vertex $y_n \not \in B_n$. Set $M_n = \phi_n(M)$ for $n \in \bN$. So $M_n$ misses only $y_n \not\in B_n$. Applying Lemma \ref{miss-none-in-balls} to $(M_n)_{n \in \bN}$ we obtain a matching that misses no vertex in~$\Gha$.
\end{proof}

\end{document}